\documentclass[10pt]{amsart}

\usepackage{amsmath,amsthm,amsopn,amsfonts,amssymb,color}
\usepackage[colorlinks=true,linkcolor=blue,urlcolor=blue]{hyperref} 
\usepackage{cancel}
 
\usepackage{graphicx}
\usepackage{float}
\usepackage{amsmath}
\usepackage{fancyhdr}
\usepackage{epstopdf}
\usepackage{amsfonts}
\usepackage[T1]{fontenc}
\usepackage[normalem]{ulem}
\usepackage{booktabs}
\usepackage[USenglish]{babel}
\usepackage{cancel}
\usepackage{times}
\usepackage{mathrsfs}
\usepackage{multirow,multicol}
\usepackage{amsfonts}
\usepackage{amsthm}
\usepackage{amsmath}
\newtheorem{teo}{Theorem}[section]

\newtheorem{defin}[teo]{Definition}
\newtheorem{remark}[teo]{Remark}
\newtheorem{cor}[teo]{Corollary}
\newtheorem{lemma}[teo]{Lemma}

\def\elle#1{L^{#1}(\Omega)}

\def\w{H_0^{1}(\Omega)}

\def\norma#1#2{\|#1\|_{\lower 4pt \hbox{$\scriptstyle #2$}}}

\def\Omegat{\tilde{\Omega}}

\def\finedim
\def\gw{G_{\tilde{k}}(w_n)}

\def\elle#1{L^{#1}(\Omega)}

\def\w{W_0^{1,2}(\Omega)}

\def\w1{W_0^{1,1}(\Omega)}

\def\dys{\displaystyle}

\def\w{H_0^{1}(\Omega)}

\def\be{\begin{equation}}
\def\ee{\end{equation}}
\def\bc{\begin{cases}}
\def\ec{\end{cases}}

\def\be{\begin{equation}}
\def\ee{\end{equation}}
 \usepackage{color}

 \numberwithin{equation}{section}

\newcommand{\EEE}{\color{black}}

\newcommand{\USC}{\text{\rm USC}}
\newcommand{\LSC}{\text{\rm LSC}}
\newcommand{\Rz}{\mathbb{R}}

 \usepackage{color}
 \numberwithin{equation}{section}
\usepackage{todonotes}

\title[Viscosity solutions for nonlocal equations]{
 Viscosity solutions for nonlocal equations \\with space-dependent operators}
\author[S. Buccheri] {Stefano Buccheri}
\address[Stefano Buccheri]{Faculty of Mathematics, University of Vienna, 
Oskar-Morgenstern-Platz 1, A-1090 Vienna, Austria}
\email{stefano.buccheri@univie.ac.at}
\author[U. Stefanelli] {Ulisse Stefanelli} 
\address[Ulisse Stefanelli]{Faculty of Mathematics, University of Vienna, 
Oskar-Morgenstern-Platz 1, A-1090 Vienna, Austria,\,\&
Vienna Research Platform on Accelerating Photoreaction Discovery, University of Vienna, W\"ahringerstrasse 17, A-1090 Vienna, Austria,\,\&
Istituto di Matematica Applicata e Tecnologie Informatiche E. Magenes,
via Ferrata 1, I-27100 Pavia, Italy.
}
\email{ulisse.stefanelli@univie.ac.at}

\subjclass[2010]{}

\begin{document}

\begin{abstract}
  We consider a class of elliptic and parabolic problems,
  featuring a specific nonlocal operator of fractional-laplacian
  type, where integration is taken on variable domains. 
  Both elliptic and parabolic problems are proved to be uniquely solvable in
  the viscosity sense.   
% The viscosity solution of the elliptic problem is proved to be
%  \MMM H\"older continuous \EEE[ Questa frase la eliminerei/cambierei soprattutto se ci teniamo le ipotesi attuali. Non proviamo nessun nuovo risultato di regolarita' in definitiva, ne usiamo uno gia' esistente\EEE]. In addition, we study the corresponding eigenvalue
%  problems. 
Moreover, some spectral properties of the elliptic
operator are investigated, proving existence and simplicity of the first eigenvalue. Eventually, parabolic solutions are proven to converge to
  the corresponding limiting elliptic solution in the long-time limit.  
\end{abstract}

\maketitle
%\tableofcontents
\section{Introduction}
 The study of PDE problems driven by nonlocal operators is
attracting an ever growing attention. This is in part 
motivated by the great relevance of such operators in applications, 
among which L\'evy processes, differential games, and image
processing, just to mention a few. %,  but also by  the richness of
% the associated  mathematical theory.
The paramount
example %in this class
of nonlocal operator  is the {\it
  fractional laplacian}, which can be defined in the following %. One way to define it is by means of a singular integral to be
% meant in the
principal-value sense
\be\label{def}
(-\Delta)_ {s} u (x) =
p.v.\int_{\mathbb{R}^N}\frac{u(x)-u(y)}{|x-y|^{N+2s}}dy \quad
\text{ for} \  s\in(0,1).
\ee
 In this paper we focus on elliptic and parabolic
problems featuring a localized version of the classical
fractional-laplacian operator, namely,  
\be\label{10-6bis0}
u(x)\mapsto p.v.  \int_{ 
  \Omega(x)}  \frac{u(x)-u(y)}{|x-y|^{N+2s}} dy.
\ee
In contrast with the classical fractional laplacian, the main feature
in \eqref{10-6bis0} is that integration is taken with respect to a $x$-dependent bounded
set $\Omega(x)$. Such a modification is inspired by the analysis of
the hydrodynamic limit of kinetic
equations \cite{pedro} and of peridynamics \cite{Silling00}.
We comment on these connection in Subsection \ref{sec:connections}
below.

In order to specify our setting further, let us recall that the
homogeneous Dirichlet problem associated with the classical fractional laplacian $(-\Delta)_ {s}$\eqref{def} in a
bounded set $\Omega\subset \mathbb{R}^N$ requires to prescribe
the nonlocal boundary condition $u=0$ on  $\mathbb{R}^N\setminus
\Omega$. Under such condition the fractional laplacian can be written as 
\begin{align}\label{decomposition}
&(-\Delta)_ {s} u (x) = p.v.\int_{\mathbb{R}^N}\frac{u(x)-u(y)\chi_{\Omega}}{|x-y|^{N+2s}}dy=p.v.\int_{\Omega}\frac{u(x)-u(y)}{|x-y|^{N+2s}}dy+u(x)\int_{\mathbb{R}^N\setminus\Omega}\frac{1}{|x-y|^{N+2s}}dy
  \\
  &\quad =:(-\Delta)^{\Omega}_ {s} u (x) +k(x) u(x)\nonumber,
  \end{align} where $\chi_\Omega$ indicates the characteristic
  function of $\Omega$. 
The nonlocal operator $(-\Delta)^{\Omega}_ {s}$ is usually
called {\it regional fractional laplacian}. By indicating with
$d(x)$ the distance of $x\in \Omega$ from the boundary $\partial
\Omega$,  under mild regularity assumption
on $\partial\Omega$  the function $k(x)$ can be proved to satisfy
\begin{equation}\label{killing}
 \frac{\alpha}{d(x)^{2s}}\le k(x)\le  \frac{\beta}{d(x)^{2s}},
\end{equation}
 for some $0<\alpha<\beta$. We provide a proof of this property
 in % This property is very well known and we add a proof in
 Lemma \ref{hbeha}. % below for completeness\EEE.
 Differently from  the classical fractional laplacian  \eqref{def}, 
the regional laplacian $(-\Delta)^{\Omega}_ {s}$ acts on
functions defined in $\Omega$. Correspondingly, boundary
conditions for the homogeneous Dirichlet problem for
$(-\Delta)^{\Omega}_ {s}$ can be directly prescribed % and, typically, the associated
% boundary conditions are prescribed
on $\partial\Omega$. %As one can imagine looking at
% \eqref{decomposition}, t
%The two operators $(-\Delta)_ {s}$ and $(-\Delta)^{\Omega}_ {s}$ \MMM
%hence differ %from each other
%especially in the vicinity of\footnote Sembra che la differenza
%sia dovuta al fatto che le condizioni al bordo siano nel primo caso
%non locali e nel secondo caso locali. Ma non e' cosi'. Il punto e' la
%presenza o meno di un termine tipo $k(x)$. Frase alternativa :
The two operators $(-\Delta)_ {s}$ and $(-\Delta)^{\Omega}_ {s}$
differ especially in the vicinity of the boundary, as one can expect
looking at \eqref{decomposition}. 
the boundary. While for any $s\in(0,1)$ the solution of the homogeneous Dirichlet problem
associated with \eqref{def} behaves as $d(x)^s$ as $x$
approaches $\partial \Omega$, the one associated to
$(-\Delta)^{\Omega}_ {s}$ goes to zero as $d(x)^{2s-1}$ for
$s\in(1/2,1)$. In fact, for $s\in(0,1/2]$, the Dirichlet
problem associated to the regional laplacian is not well-defined, independently of the regularity of $\partial \Omega$. 
This can be explained via trace theory  (the trace operator
exists only if $s>1/2$, see \cite{tar}, \cite{fall}), or in term of
the probabilistic process associated to $(-\Delta)^{\Omega}_ {s}$
(such a process reaches the boundary only if $s>1/2$, see
for instance \cite{bog}). Roughly speaking, we can say that the term
$k(x) u(x)$ regularizes the operator $(-\Delta)^{\Omega}_ {s}$ close
to the boundary, by forcing a quantified convergence of solution
to zero on approaching $\partial \Omega$. The reader can find
more on the relation between $(-\Delta)_ {s}$ and $(-\Delta)^{\Omega}_{s}$
in the recent survey \cite{alot} and in the references therein.

Inspired by position \eqref{10-6bis0}, by 
decomposition   \eqref{decomposition},  and by property
\eqref{killing}, we aim at
considering  more general nonlocal operators of the following form
\be\label{10-6}
h(x) u (x)+\mathcal{L}_{s} (\Omega(x),u (x)),
\ee
 where $\mathcal{L}_{s} (\Omega(x) ,u (x_0))$ is defined as
\be\label{10-6bis}
\mathcal{L}_{s} (\Omega(x),u (x_0))=p.v.  \int_{y\in
  \Omega(x)}  \frac{u(x_0)-u(y)}{|x-y|^{N+2s}} dy. \EEE
\ee
In the following, we often use the  change of coordinates
$z=x-y$. In this  new reference frame we will write 
\be \label{omegatxb}
\Omegat(x)=\{z\in\mathbb{R}^N \ : \ x-z\in \Omega(x)\}.
\ee
In \eqref{10-6}, the function $h(x)\in C(\Omega)$ is assumed to be given and to fulfill
\begin{equation}\label{alpha}
0<  \frac{\alpha}{d(x)^{2s}}\le h(x)\le  \frac{\beta}{d(x)^{2s}},  \ \ \ \mbox{with} \ \ \ 0<\alpha\le\beta,
\end{equation}
and the set-valued function $x\to\Omega(x)\subset\Omega$ is assumed to satisfy 
%\be
\begin{align}\label{continuita}
& \quad \forall x \in \Omega:\quad  \lim_{y\to x}|\Omega(y)\triangle\Omega(x)|=0,\\
 & \quad \exists \ \zeta\in \big(0,1/2\big) , \   \forall x \in \Omega :\quad \Omegat(x)\cap B_{r}(0)=\Sigma\cap B_{r}(0) \mbox{ for all } r\le \zeta d(x),\label{ostationary}
\end{align}
%\ee
for  some given open set $\Sigma$ such that
\be\label{Sigmadef}
\Sigma=-\Sigma \quad \mbox{and} \quad \left|\Sigma\cap
  \big(B_{2r}(0)\setminus B_r(0)\big)\right|\ge q|B_{2r}(0)\setminus
B_r(0)| \mbox{  for any } r>0,
\ee
with $q\in(0,1)$.  Note that the symmetry of $\Sigma$ is
required  for the well-definiteness of the  % in order to have
% the
operator %well defined for
on smooth functions and that assumption \eqref{continuita}
ensure that operator $\mathcal{L}_{s}$ from  \eqref{10-6bis} is
 continuous with respect to his arguments. For a more extensive
discussion on the hypothesis look at Section \ref{existingliterature}.

Our first goal is to show that the {\it elliptic problem} related to
the nonlocal operator \eqref{10-6} given by
\begin{equation}\label{10:17}
\begin{cases}
h(x)u (x)+\mathcal{L}_{s} (\Omega(x),u (x))= f(x) \qquad & \mbox{in } \Omega,\\
 u(x) = 0 & \mbox{on } \partial\Omega ,\\
\end{cases}
\end{equation}
is uniquely solvable, for any $f\in C(\Omega)$ such that
\be\label{fcon}
\exists \  \eta_f\in (0,2s), C>0 \ : \ |f(x)|d(x)^{2s- \eta_f}\le C.
\ee 
% The main features of \eqref{10-6} are, on one hand, that the integration domain in \eqref{10-6bis} depends on $x$, and, on the other, that the function $h(x)$ explodes as $d(x)^{2s}$ as $x$ approaches $\partial \Omega$, where $s$ is the fractional order of \eqref{10-6bis}. The first property allows us to restrict the nonlocal interaction to subregions of $\Omega$, while the second one assures that the Dirichlet problem \eqref{10:17} is well-posed. Having in mind the decomposition \eqref{decomposition}, we can heuristically say that the Dirichlet boundary condition is embedded directly in the equation of \eqref{10:17}, through assumptions \eqref{alpha} and~\eqref{fcon}.\\
%  
  In the current setting it is natural to address
problem \eqref{10:17} in the viscosity sense. %with the
% viscosity solutions approach.
Indeed, assumptions \eqref{continuita}-\eqref{Sigmadef} imply  that the
operator \eqref{10-6} satisfies the comparison principle. Moreover the
growth conditions \eqref{alpha} and~\eqref{fcon} allow us to
build a barrier for  problem \eqref{10:17} of the form
\be\label{barrierintro}
C_\alpha d(x)^{ \eta} \ \ \ \mbox{for some small } \eta>0.
\ee
 % at the beginning of the
                                % section). Putting together this two
                                % ingredients ( by combining the comparison principle and existence of a barrier,
Having comparison and barriers at hand, we can implement the Perron method and prove in Theorem \ref{teide} the existence of a
 unique viscosity solution for \eqref{10:17}. 
Note that $C_\alpha$ degenerates with $\alpha$. In particular,
 %We explicitly pointed out the dependence on $\alpha$ for the
% constant $C_\alpha$, to stress that,
if the term $h(x)$ degenerates %vanishes or deteriorates
% close to
at the boundary,  solvability of problem \eqref{10:17}  may fail,
at least for  $s\in(0,1/2]$, %  we do not expect problem \eqref{10:17}
                             %  to be solvable (at least for
                             %  $s\in(0,\frac12]$,
see the above discussion.

A second aim of our paper is to address the spectral
properties of the operator \eqref{10-6}. We focus on the first
eigenvalue associated to \eqref{10-6} with homogeneous Dirichlet
boundary conditions and  we show that there exists a unique
$\overline{\lambda}>0$ such that the problem
 \begin{equation}\label{10:17ter}
\begin{cases}
h(x)v (x)+\mathcal{L}_{s} (\Omega(x),v (x))= \overline{\lambda} v  (x)\qquad & \mbox{in } \Omega,\\
 v(x) = 0 & \mbox{on } \partial\Omega,\\
\end{cases}
\end{equation}
admits a strictly positive viscosity  solution. Moreover
such a solution is unique, up to multiplication by
constants. % In other words, $\overline{\lambda}$ is the first
% eigenvalue of \eqref{10:17ter}.
As our setting in nonvariational, the characterization of such
first %Since we are not in a variational setting, we characterize
% such
eigenvalue follows %being inspired by
the approach of the seminal work \cite{bere}. More precisely, we
define $\overline{\lambda}$ as the supremum of the values $\lambda\in \mathbb{R}$ such that the problem
\be\label{refineed}
\begin{cases}
h(x)u(x)+\mathcal{L}_{s} (\Omega(x),u (x))= \lambda u (x) + f (x)\EEE\qquad & \mbox{in } \Omega,\\
 u(x) = 0 & \mbox{on } \partial\Omega\, ,\\
 u(x) > 0 & \mbox{in } \Omega\, ,\\
\end{cases}
\ee
admits a solution, for  some given nonnegative and 
       nontrivial $f\in C(\Omega)$  satisfying
        \eqref{fcon}, see Theorems \ref{mussaka}-\ref{helmholtz}.
         As we shall see, this characterization does not depend
        on the particular choice of $f$. Such a definition is slightly
        different from the usual one (see for instance \cite{biri},
        \cite{busca}, \cite{quaas} and \cite{biswas} for the same
        approach in both local and nonlocal settings), being
        based on the concept of solution instead of that of
        supersolution. %\MMM This allows us to prove stronger
       % properties on $\overline{\lambda}$, including its
         %             uniqueness.\footnote{ Scritto cosi' sembra che sia necessario usare questa caratterizzazione diversa da quella standard. In relata' e' solo piu' comoda. Alternative: 1)semplicemente cancellare la frase; 2) We prefer this alternative since is more versatile in proving some properties of $\overline{\lambda}$, as its uniqueness. \EEE  } \EEE %  makes this concept better suited
        %We prefer this alternative since is more versatile in proving some properties of $\overline{\lambda}$, as its uniqueness.\\

A further focus of this paper is the study of the
evolutionary counterpart of \eqref{10:17}, namely, the {\it
  parabolic problem} 
\begin{equation}\label{10:17bis}
\begin{cases}
\partial_t u (t,x) (x)+h(x)u
(x)+\mathcal{L}_{s} (\Omega (t,x),u (t,x))=
f(t,  x) \qquad & \mbox{in } (0,T)\times\Omega,\\
 u(t,x) = 0 & \mbox{on } (0,T)\times\partial\Omega ,\\
 u(0,x) = u_0(x) & \mbox{on }   \Omega.
\end{cases}
\end{equation}
We prove existence and uniqueness of a global-in time viscosity
 solution to problem \eqref{10:17bis}. See Theorem \ref{existence}
below for the precise statement in a more general setting, where
a time-dependent version of assumption \eqref{ostationary} is
considered. The behavior of the solution \EEE %\EEE After that, we
                                %address the question of the limiting
                                %behavior of the solution
$u(x,t)$ for large times is then addressed in Theorem
\ref{lalaguna}:  Taking advantage of the characterization of $\overline{\lambda}$ we prove that for any $\lambda <\overline{\lambda}$ there exists $C_{\lambda}$ such that
\[
|u(t,x)-u(x)|\le C_{\lambda} \overline{\varphi} (x) e^{-\lambda t},
\]
where $\overline{\varphi}$ is the normalized positive eigenfunction associated to $\overline{\lambda}$, $u(x)$ is the solution of the elliptic problem
\eqref{10:17} and $u(t,x)$ solves the parabolic problem
\eqref{10:17bis} for the same time-independent forcing $f(x)$. \EEE

\subsection{Relation with applications}\label{sec:connections}
 As mentioned, the specific form of operator $\mathcal{L}_{s}$, in particular the
 dependence of the integration domain $\Omega(x)$ on $x$, occurs in
 connection with different applications.

 A first occurrence of operators of the type of $\mathcal{L}_{s}$ is
 the study of the hydrodynamic limit of 
collisional kinetic equations with a heavy-tailed thermodynamic
equilibrium. When posed in the whole space, the reference nonlocal
operator in the limit is the classical fractional laplacian
\eqref{def}, see
\cite{melletbis}. The restriction of the dynamics to a  bounded domain with a zero
inflow condition at the boundary asks for considering
\eqref{10-6bis0} instead \cite{pedro}. In this connection,
$\Omega(x)$ is defined to be the largest star-shaped set centered
at $x\in\Omega$ and contained in $\Omega$. The heuristics for this
choice is that a particle centered at $x$ is allowed to move along
straight paths and is removed from the system as soon as it reaches
the boundary. Hence, the possible interaction range of a particle
sitting at $x$ is exactly  $\Omega(x)$. % 
%  da qui, 
% A first motivation for the study of \eqref{10-6} comes from the hydrodynamic limit of some kinetic equation on bounded domain with non-standard velocity distribution. The rough idea is to consider collisional kinetic equation with a heavy-tailed thermodynamic equilibrium and study the asymptotic of a properly rescaled problem. If one set the kitetic equation in the whole space, the resulting hydrodinamic limit will lead to \eqref{def} (see \cite{melletbis}).
% Instead, the authors of \cite{pedro} address the case of kinetic equation on a bounded domain with a zero inflow condition at the boundary. 
 In particular, the  resulting hydrodynamic  limit  under
homogeneous Dirichlet conditions features the nonlocal functional 
\be\label{start}
a(x)u (x)+ p.v.\int_{\Omega(x)}
\frac{u(x)-u(y)}{|x-y|^{N+2s}} dy =: a(x)u (x) +
(-\Delta)^\star_s u(x),  \ \ \ s\in(0,1).  %(-\Delta)_{s}^{\star}u (x)
\ee
% where the fractional operator is formally defined as
% \[
% (-\Delta)_{s}^{\star}u  (x)= p.v.\int_{\Omega(x)}  \frac{u(x)-u(y)}{|x-y|^{N+2s}} dy,  \ \ \ s\in(0,1),%\ \ \ \ \ \frac{\lambda}{|x-y|^{N+\alpha}}\le K(x,y)\le \frac{\Lambda}{|x-y|^{N+\alpha}}
% \]
 Here, the function  %with %$S(x)$ the largest star shaped set
                             %centered at $x\in\Omega$ and contained
                             %in $\Omega$, and the function
$a(x)$
has the following specific form
\be\label{15:15}
a(x)=\int_{\mathbb{R}^N}\frac{1}{|y|^{N+2s}}e^{-\frac{d(x,\sigma(y))}{|y|}}dy,
\ee
with $d(x,\sigma(y))$  being  the length of the segment joining $x\in \Omega$
with the closest intersection point between $\partial \Omega$ and the
ray starting from $x$ with direction $\sigma(y)
=\frac{y}{|y|}$.  Clearly,  if $\Omega$ is convex  one
has that $\Omega(x)\equiv \Omega$ for all $x$  and the function
$a(x)$ coincides, up to a constant, with the function $k(x)$ of
\eqref{decomposition} (see Lemma \ref{acca}).  In this case  we recover exactly
(up to a constant) the operator in \eqref{decomposition}.   Note
however that  
even in the case of a nonconvex domain $\Omega$, the function  $a(x)$ satisfies the condition
\eqref{alpha} (see Lemma \ref{hbeha}). 

 A second context where nonlocal operators of the type of
\eqref{10-6bis} arise is that of {\it peridynamics} \cite{Silling00}. This is a
nonlocal mechanical theory, based on the
formulation of equilibrium systems in integral instead of differential
terms.
Forces acting on the material point $x$ are obtained as a combined effect of interactions with other
points in a given neighborhood. This results in an integral featuring a radial
weight which modulates the influence of nearby points in terms of
their distance \cite{Emmrich,survey}. A reference nonlocal operator in
this connection is
\be\label{peridyn}
u(x) \mapsto %(-\Delta)_s^{\rho}u(x)=
p.v.\int_{  B_{\rho}(x)}  \frac{u(x)-u(y)}{|x-y|^{N+2s}} dy.
\ee
Here, $ B_{\rho}(x)$ is the call of radius $\rho>0$ centered at
$x$. In particular, the parameter $\rho$ measures the interaction
range. Such operators have used to approximate the
fractional laplacian in numerical simulations (see \cite{duo} and the
reference therein), note also the parametric analysis in  
\cite{burkovska}. The operator $\mathcal{L}_s$ in \eqref{10-6bis} corresponds hence to a natural
generalization of the latter to the case of an interaction range which
varies  along the body, as could be the case in presence of a
combination different material systems. This would correspond to
choosing a varying $\rho(x)$.

\subsection{Existing literature}\label{existingliterature}
 To our knowledge, operator \eqref{10-6} has not been studied
yet.    %We point out that, up to our knowledge, an operator with
            %the specific structure of \eqref{10-6} has not been
            %addressed in the literature yet.
Despite its simple structure when compared to the general operators
 usually allowed in the fully nonlinear setting, most of the 
available tools seem not to directly apply. 

 In this section, we aim at presenting a brief account of the
literature in order to put our contribution in context.  %Now let
                                %us place our work within the existing
                                %results. After
Following  the seminal work \cite{CIL}, the existence theory of
viscosity solutions, through comparison principle, barriers, and
Perron method, has been generalized  to  a large class of elliptic and parabolic integral differential equations, see for instance \cite{barimb}, \cite{bci}, \cite{caff}, \cite{cd}, \cite{cdbis} and \cite{jk}.

% The most representative contributions for the
Comparison principles  for nonlocal problems  in the viscosity
setting  can be found in  \cite{barimb}, see also
\cite{jk}. One of the key structural assumptions of these works reads,
in our notation,  
\be\label{bbarles}
\int_{\mathbb{R}^N}|\chi_{\Omegat(x)}-\chi_{\Omegat(y)}||z|^{2-N-2s}dz\le c|x-y|^2
\ee
for some positive constant $c>0$ (see assumption (35) in
\cite{barimb}). This type of condition allows the authors to implement
the  variable-doubling  strategy of \cite{CIL}  for  a
large class of  operators of so-called L\'evy-Ito type.  This is not expected here, % making this approach %Anyway such a result is valid
                                %for uniformly continuous functions
                                %and then it seems
% not suitable for %our purposes (namely,
%the implementation of the Perron method.  Indeed, 
%Our set of assumptions does not imply \eqref{bbarles} and ideed such a condition may fail, as the following simple argument shows.  Let $s\to\nu(s)$ any real valued function continuous at $0$ such that $\nu(0)=0$, assume that $\Omegat(x)=B_1(0)$ for some $x\in\Omega$, and that $\Omegat(y)=B_{(\rho(|x-y|))}(0)$ in a small neighborhood of $x$, with $\rho(s)=(1-\nu(s))^{\frac1{2-2s}}$. Then, the integral in the left hand side above is equal to $\omega_N \nu(|x-y|)$.\\
%With respect to \cite{barimb}, we point out that
 for  our set of assumptions does not imply \eqref{bbarles} %and
                                % such uniform continuity \EEE may
                                %fail,
 as the following simple argument shows: let
$\Omega$ be the unitary ball centered at the origin and $\Omegat(x)=B_{\rho(x)}$ for any
$x\in\Omega$ with $\rho(x)=d(x)^{ 1/(2-2s)}$. Then, for any $x,y\in\Omega$ we get
\[
\int_{\mathbb{R}^N}|\chi_{\Omegat(x)}-\chi_{\Omegat(y)}||z|^{2-N-2s}dz=\frac{\omega_N}{2-2s}|\rho(x)^{2-2s}-\rho(y)^{2-2s}|=\frac{\omega_N}{2-2s}|x-y|.
\]
Our alternative strategy is to assume that the operator is somehow
translation invariant \emph{close} to the singularity,  with 
this property  degenerating while \EEE approaching
$\partial\Omega$, see assumption \eqref{ostationary} above. This
allows for some cancellations that bypass the problem of the
singularity of the kernel. Instead of doubling variables, we use the
inf/sup-convolution  technique  to prove that sum of viscosity subsolution is still a viscosity subsolution. Eventually, we
also quote the interesting comparison result in
\cite{gms}. There,
the doubling of variable is combined with an optimal-transport
argument. Such an approach, however, requires the uniform
continuity of solutions, %is viable if solutions are uniformly
                         %continuous 
 and it is not clear how to adapt it for general viscosity solutions.

As far as the construction of barriers is concerned, notice that 
the typical  difficulty is to estimate from below a term of the type
\[
\int_{\mathbb{R}^N\setminus\Omega}\frac{1}{|x-y|^{N+2s}}dy,
\]
 which  requires some regularity on the boundary of
$\Omega$. Of course  we refer here to the standard case of the
fractional laplacian,  %this refers to the simple case of the
% fractional laplacian,
but the same idea can be  extended to more general operators, see 
% seen, for instance, in
 \cite[Lemma 1]{bci}.  In our case,  since we impose a priori condition \eqref{alpha}, we can actually deal with any open domain, paying the price of a poor control on the decay of the solution close to the boundary, see \eqref{barrierintro}.

For an alternative approach to the Perron method,  not relying on
the 
comparison principle and therefore produces discontinuous viscosity
solutions, we  refer   the interested reader to \cite{mou}.

For a comprehensive overview on the numerous contributions to the
regularity theory for viscosity solutions to nonlocal elliptic and
parabolic equations, we address the reader to the  rather detailed
 introductions of \cite{schwabsil} and \cite{krs}. Here, we provide a small overview on some results more closely related to our work.
The first fully PDE-oriented result about H\"older regularity for elliptic nonlocal operators has been obtained in \cite{silve}. 
A drawback of this approach is that it does not allow to
consider the limit $s\to1$. The first H\"older estimate
which is robust enough to pass to the limit as $s\to1$ has been
obtained in \cite{caff} (see also \cite{caffbis}) and then generalized
to parabolic equation in \cite{cd} and \cite{cdbis}. All these results
apply to operators whose kernel is pointwise controlled from above and
from below by the one of $(-\Delta)_s$. For results where such  pointwise control is not avaliable, we refer again to
\cite{schwabsil} and \cite{krs}. More in detail, our condition
\eqref{Sigmadef} is a simplified version of assumption (A3) in
\cite{schwabsil}. Condition \eqref{Sigmadef} allows us to
deduce an interior regularity estimate, in the spirit of the more general  \cite[Theorem 4.6]{mou}.
%Let us stress that we are not going to use that the such estimates are stable as $s\to1$. Then would be interesting to replace \eqref{Sigmadef} with the simpler density assumption
%\[
%\Sigma=-\Sigma \quad \mbox{and} \quad \left|\Sigma\cap  B_r\big)\right|\ge q|B_{r}| \mbox{ for any } r>0
%\]
%and prove an interior regularity result for the associated operator.
 \\

As mentioned, we define the first eigenvalue of \eqref{10-6} 
following the approach in   \cite{bere}.
  The advantage of this approach is that it  is independent of
   the variational structure of the %considered
  operator  by directly relying on % and rather takes advantage of
  % the relation among
  the  maximum principle,  as well as on the  existence of
  positive (super-)solutions. For this reason it
  has been fruitfully used in the framework of viscosity solution for
  second order fully nonlinear differential equations, see for
  instance \cite{biri}, \cite{busca}, and \cite{QS}
  
 An early   result related to  eigenvalues of  nonlocal
operators with singular kernel is  in  \cite{bego}  where
existence issues in presence of lower order terms are tackled. 
% .  Here,  the authors focus on existence issues for
% fractional equations with nonlinear lower order terms and they need
% to derive some properties of the principal eigenvalue along their
% arguments.
A  closer   reference is \cite{DQT}, where the principal
eigenvalues of some fractional nonlinear equations, with inf/sup
structure are studied. In this paper the authors prove, among other
results, existence and simplicity of principal eigenvalues together
with some isolation property and the antimaximum principle. Other
results following the same line of investigation can be found in
\cite{quaas} and \cite{biswas}. We point out that the operators
considered in these works are just  positively  homogeneous
(i.e. $\mathcal{L}(u)\neq-\mathcal{L}(-u)$),  which gives rise to
the existence of 
two principal {\it half-eigenvalues}, namely corresponding to
differently signed
eigenfunctions. We are not concerned with this phenomenon here. % 
                                % This phenomenon does not concern our
                                % case of study and we refer the
                                % interest reader to the introduction
                                % of aforementioned papers for more
                                % details.
A common tool used in the previous works to prove existence of
eingenvalues is a nonlinear version of the Krein-Rutmann theorem for
compact operators, see again \cite{quaas} and \cite{biswas} and 
the references  therein.  Let us also mention the recent %We
                                %also quote other two recent
                                %contributions,
 \cite{birga} and \cite{bismod}, that deal with  a different  kind of fractional operators.

To prove existence of the first eigenvalue, we follow a direct
approach based on the approximation of problem
\eqref{10:17ter}. Unlike the previously quoted papers,
we do not resort to a %have any 
global regularity result to deduce either the compactness of the operator or
the uniform convergence of approximating solutions. Instead, we
combine the  so-called {\it half-relaxed-limit} 
method, a version of the {\it refined maximum principle} and interior regularity result of \cite{schwabsil,mou}. The
half-relaxed limit method is a powerful tool to pass to the limit with
no other regularity then uniform boundedness, see for instance
\cite{barles} and  the  references therein. In general, the price to allow such generality is 
to handle discontinuous viscosity solutions.  We however avoind
this, for  % however we avoid this technicality since
we are able to prove a refined maximum principle for \eqref{refineed}
in the  spirit  of \cite{bere}. This, in turns, provides a
comparison between sub and super solutions,  eventually ensuring
\EEE continuity.

Due to  particular  structure of  our nonlocal 
operator, a key ingredient in our proof is a restriction procedure for
\eqref{10-6bis} in subdomains of $\Omega$,  which is where 
the density assumption \eqref{Sigmadef} is needed. Heuristically
speaking, such assumption forces the operator to \emph{look the same
  at every scale},  as for the kernel singularity and the %meaning
                                %that not only the singularity of the
                                %kernel is preserved but also the
 behavior \eqref{alpha} of the term $h$ (see Lemma \ref{localization}). 
 Eventually, our refined maximum principle allows us to show
 uniqueness of the first eigenvalue and its simplicity.

 A  general reference for the long-time behavior of solutions to
 nonlocal parabolic equation is the monograph \cite{bookjulio}. We also refer to \cite{berebis} and to the already mentioned  \cite{quaas}.
 In the latter work, a fractional operator with drift term is considered and the viscosity solution of the associated homogeneous 
 initial-boundary value problem is proved to converge to zero in the
 large-time limit.
 
 Here, we consider a non-homogeneous parabolic equation with initial datum and homogeneous Dirichled boundary condition.
 Moreover, we allow the coefficients of our operator to be time-dependent. Under suitable assumption on this time dependency, 
we prove that the parabolic solution converges to the stationary one
exponentially in time. The exponenential rate of convergence depends on the principal eigenvalue.

\section{Statement of the main results}

In this section, we %We
collect %in this section
our notation and state our main results. % We focuson viscosity solutions of the elliptic problem \eqref{10:17} and ofits parabolic counterpart \eqref{parabolic}. 

Given any $D\subset
\mathbb{R}^{M}$,  we indicate
\EEE upper and lower semicontinuous functions on $D$ as
\[
\USC(D)=\left\{u:D\to \mathbb{R} \ : \ \limsup_{z \to z_0}u(z)\le u(z_0)\right\},
\]
\[
\LSC(D)=\left\{u:D\to \mathbb{R} \ : \ \liminf_{z\to z_0}u(z)\ge u(z_0)\right\}.
\]
We also write $\USC_b(D)$ and $\LSC_b(D)$ for the set of upper and
lower semicontinuous functions that are bounded. Given a function
$u:D\to \mathbb R$ we indicate \EEE its {\it upper and lower
semicontinuous envelopes}  as
\[
u^*(z_0)=\limsup_{z\to z_0}u(z) \quad \text{and} \quad
u_*(z_0)=\liminf_{z\to z_0}u(z),
\]
respectively.  \EEE

%We now specify thisnotion in the parabolic case, by assuming $T\in (0,\infty]$. \EEE

%Let us now introduce the notions of solutions that we consider in the sequel. Let us set $T\in (0,\infty]$.
\begin{defin}[Viscosity solutions]\label{def1}
  Elliptic case:
  We say that $u\in \USC_b(  \Omega  )$ ($\in \LSC_b(  \Omega  )$) is a \emph{viscosity sub (super) solution} to  the equation
\[
h(x)u (x)+\mathcal{L}_{s} (\Omega(x),u (x))= f(x)
\]\EEE
if, whenever $x \in \Omega  $ and $\varphi\in
C^2( \Omega)$ are such that $u( x )=\varphi(x )$
and $u(y)\le \varphi(y)$ for all $y\in\Omega$, then  
\begin{equation*}
 h( x ) \varphi(x )+\mathcal{L}_{s} (\Omega(x),\varphi(x ))\le (\ge) \ f(x ). 
\end{equation*}
 Moreover $u\in C(\overline \Omega)$ is a \emph{viscosity solution} to problem \eqref{10:17} if is both sub- and supersolution and satisfies the boundary condition $u=0$ on $\partial \Omega$ pointwise.\EEE

  Parabolic case: \EEE
We say that $u\in \USC_b( (0,T)\times\Omega  )$ ($\in \LSC_b( (0,T)\times\Omega  )$) is a \emph{viscosity sub (super) solution} to  the equation 
\[
\partial_tu(t,x) +h(x) u(t,x)+\mathcal{L}_{s} (\Omega(t,x),u(t,x) )= f(t,x) \qquad \mbox{in }  (0,T)\times\Omega
\] \EEE
if, whenever $( t , x )\in (0,T)\times\Omega  $ and $\varphi\in
C^2((0,T)\times{\Omega})$ are such that $u( t , x )=\varphi( t , x )$
and $u(\tau,y)\le \varphi(\tau,y)$ for all $\tau \EEE,y\in(0,T)\times\Omega$, then  
\begin{equation*}
 \partial_t  \varphi( t , x )+ h( x ) \varphi( t , x )+\mathcal{L}_{s} (\Omega(t,x),\varphi( t , x ))\le (\ge) \ f( t , x ). 
\end{equation*}
 Moreover $u\in C([0,T)\times\overline \Omega)$ is a \emph{viscosity solution} to \eqref{parabolic} if it is both sub- and supersolution and satisfies the boundary and initial conditions
  \begin{align}
    \label{parabolic}
 \left\{   \begin{array}{ll}
    %\label{bounin1}
 u(t,x) = 0 & \mbox{on } (0,T)\times\partial\Omega\, ,\\
    u(0,x) = u_0(x) & \mbox{on }   \Omega,
           \end{array}
                      \right.
\end{align}
pointwise.\EEE
\end{defin}

% This definition can be immediately adapted to the elliptic
% case. \EEE Indeed, for any $v:\Omega\mapsto \mathbb R$, we can write,
% with a slight abuse of notation, $v(t,x)=v(x)$ and apply to this
% function the notions of sub and super solutions above. Clearly, any
% smooth function $\varphi (t,x)$ that touches from above (or below)
% $v(t,x)=v(x)$, at a given point $(t_0,x_0)$, must satisfies
% $\partial_t \varphi(t_0,x_0)=0$.

We are now in the position of stating our main results.

\begin{teo}[Well-posedness of the elliptic problem]\label{teide}
Let us assume \eqref{alpha}-\eqref{Sigmadef}, and
\eqref{fcon}. Then, problem \eqref{10:17} admits a unique viscosity
solution $u$. This  satisfies $|u(x)|\le C d(x)^{\eta}$ for some suitable 
small $\eta>0$ and large $C >0$ . 
\end{teo}
%Our second question concerns the first-eigenvalue problem associated to \eqref{10-6}.

\begin{teo}[Well-posedness of the elliptic first-eigenvalue
  problem]\label{mussaka}
Under the same assumption of Theorem \ref{teide}, there exists a
unique  $\overline{\lambda}>0$ \EEE   such that 
\begin{equation}\label{eigenproblem}
\begin{cases}
{h}(x)v (x)+{\mathcal{L}}_s (\Omega(x), v(x))=\overline{\lambda} v (x)\qquad & \mbox{in } \Omega,\\
 v(x) = 0 & \mbox{on }  \partial \Omega,
\end{cases}
\end{equation}
admits a nontrivial (strictly) positive viscosity 
solution. Such solution is unique, up to a multiplicative
constant. Moreover, the first eigenvalue can be
characterized as  $\overline{\lambda}=\sup E$, where
\begin{equation}\label{eq:E}
E=\{\lambda\in\mathbb R \ : \ \exists v\in C(\overline{\Omega}), \ v>0 \mbox{ in } \Omega \ v=0 \mbox{ on } \partial\Omega  \ \  \mbox{such that} \ \ hv+\mathcal{L}_s(v)= \lambda v+f\}.
\end{equation}
and $f\in C(\Omega)$ is any given positive and nonzero function 
satisfying \eqref{fcon}. Note in particular, that the set $E$ is
independent of $f$. 
\end{teo}
% \MMM The statement of Theorem \ref{mussaka} features the set $E_f$,
% which is specified by the following result. In particular, we prove
% that $E_f $ and is in fact 
% independent of $f$.% \footnote{
 % Non ho ben capito la frase precendente. Il fatto che $E_f$ sia indipendente dalla scelta di f segue dall'unicita' dell'autovalore. Ovvero consideriamo $E_f$ e $E_g$ e  i due relativi sup $\lambda_f$ e $\lambda_g$. Entrambi sono autovalori, ma la dimostrazione del teorema 2.3 prova che sono lo stesso numero. Alternativa: Eliminare il pedice $f$ nella definizione di $E$ e aggiungere dopo il teorema: We point out that the set $E$ does not depend on $f$. In the next result we address existence of solutions for the problem below the resonant condition.\EEE} \EEE
\begin{teo}[ Well-posedness of the elliptic problem below the first eignevalue]\label{helmholtz}
 Under the same assumption of Theorem \ref{teide}, for any
 $\lambda<\overline{\lambda}$, the problem
\begin{equation}\label{helmeq}
\begin{cases}
{h}(x)u (x)\EEE+{\mathcal{L}}_s (\Omega(x),u(x)\EEE)=\lambda u (x)\EEE+ f(x) \qquad & \mbox{in } \Omega,\\
 u(x) = 0 & \mbox{on }  \partial \Omega,
\end{cases}
\end{equation}
admits a unique viscosity solution. In particular, we have 
$(-\infty,\overline\lambda)=E$ where the set $E$ is defined in
\eqref{eq:E}. % \subset (-\infty,\overline\lambda]$% \footnote{
 % Nel Theorem \ref{helmholtz} la soluzione trovata non deve necessariamente essere positiva. Alzi se la $f$ e' negativa la soluzione lo sara' pure. Inoltre dalla dimostrazione del Theorem \ref{mussaka} si deduce che $(-\infty,\overline\lambda)= E_f$. Non so se e' il caso di dirlo esplicitamente da qualche parte.}
\end{teo}

Let us now turn to the parabolic problem. 
% Let us now focus on the parabolic version of \eqref{10-6}.
Before stating our results, we present a time-depending
generalization of the hypothesis of Theorem \ref{teide}. We
assume
% For $s\in(0,1)$ let us set
% \be\label{fracschma}
% \mathcal{L}_{s} (\Omega(t,x),u(t,x))=p.v.\int_{y\in \Omega(t,x)}  \frac{u(t,x)-u(t,y)}{|x-y|^{N+2s}} dy,
% \ee
% where, for any $(t,x)\in (0,\infty)\times \Omega$,
the set valued function $(t,x)\mapsto\Omega(t,x)\subset\Omega$ to fulfill \EEE
the following assumptions
\begin{align}\label{contpar}
& \quad \forall (t,x) \in (0,\infty)\times\Omega:  \ \ \ \lim_{(\tau,y)\to (t,x)}|\Omega(\tau,y)\triangle\Omega(t,x)|=0,\\
 & \quad \forall \ T >0, \  \exists \
\zeta\in \big(0,1/2\big) \ :\  \forall (t,x) \in (0,T)\times\Omega \ \ \  \Omegat(t,x)\cap B_{r}(0)=\Sigma\cap B_{r}(0),\label{omegax}
\end{align}
 for all $r\le \zeta d(x)$ and $\Sigma$ as in \eqref{Sigmadef}. Recall that $\tilde{\Omega}(t,x)=\{z\in\mathbb{R}^N \ : \ x-z\in \Omega(t,x)\}$. Moreover, we 
let $h\in C((0,\infty)\times\Omega)$ satisfy for $0<\alpha\le\beta$
\begin{equation}\label{alphap}  \frac{\alpha}{d(x)^{2s}}\le h(t,x)\le  \frac{\beta}{d(x)^{2s}};
\end{equation}
and we assume that $f\in C((0,\infty)\times\Omega)$, $u_0\in
C(\Omega)$, and that there exists $ \eta_1\in(0,2s)$ such that
\be\label{fconintro}
|f(x,t)|d(x)^{2s-\eta_1}\le C,
\ee
\be\label{u0con}
|u_0(x)|d(x)^{-\eta_1}\le C.
\ee

\begin{teo}[Well-posedness of the parabolic
  problem] \label{existence}
  Let us fix $T\in(0,\infty]$. Under assumptions \eqref{Sigmadef}, 
  \eqref{contpar}-\eqref{u0con} there exists a unique viscosity solution
  \EEE $u\in C([0,T)\times\overline\Omega)\cap
  L^{\infty}([0,T)\times \Omega)$  of \EEE
  \begin{align}
    \label{parabolic}
 \left\{   \begin{array}{ll}
\partial_tu(t,x) +h(t , x) u(t,x)+\mathcal{L}_{s} (\Omega(t,x),u(t,x) )= f(t,x) \qquad &\mbox{in }  (0,T)\times\Omega,\\
    %\label{bounin1}
 u(t,x) = 0 & \mbox{on } (0,T)\times\partial\Omega\, ,\\
    u(0,x) = u_0(x) & \mbox{on }   \Omega.
           \end{array}
                      \right.
\end{align}
%Moreover such a solution is unique.
\end{teo}

Finally, we address the asymptotic behavior of the solution provided
by Theorem \ref{existence} as $T \to \infty$. In order to do 
that, we require that all the time-dependent data in  \eqref{parabolic} suitably converge to their stationary
counterparts in \eqref{10:17}. In particular, we assume that, for some $\eta_2\in(0,2s)$ and $\lambda<\overline{\lambda}$,
\begin{align}\label{decaydata}
&\left(|f(t,x)-f(x)|+|h(t,x)-h(x)|\right)d(x)^{2s-\eta_2}e^{\lambda
  t}\le C_1,\\
  \label{decayomega}
&|\Omega(t,x)\Delta\Omega(x)|d(x)^{-N-\eta_2}e^{\lambda t}\le C_2.
\end{align}
Assumption \eqref{omegax} must also be
strengthen \EEE as follows
\be\label{threeprime}
  \exists \ \zeta\in \big(0,1/2\big),\  \forall (t,x) \in I\times\Omega: \ \ \  \Omegat(t,x)\cap B_{r}(0)=\Sigma\cap B_{r}(0)
\EEE .
\ee
\begin{teo}[Long-time behavior] \label{lalaguna}
Let us assume \eqref{alpha}-\eqref{ostationary},
\eqref{contpar}-\eqref{u0con}, and
\eqref{decaydata}-\eqref{threeprime}. Let $u$ \EEE be the unique
viscosity solution of the parabolic problem \eqref{parabolic}
on $(0,\infty)$, $v$ be the unique solution
of the elliptic  problem \eqref{10:17},  and let $\lambda$ fulfill \eqref{decaydata}.
 Then, there exists $C=C(\lambda)>0 $ such that $
|u(x,t)- v(x)|\le C d(x)^{\eta} e^{-\lambda t}$ for
some small $\eta>0$.
\end{teo}
%
%\begin{teo} Let us assume that $h(t,x)$ and  $h(x)$ satisfy \eqref{alphap} and \eqref{alpha} respectively, and that $\{\Omega(t,x)\}, \{\Omega(x)\}$  satisfy \eqref{omegax} and \eqref{ostationary} respectively. Let $u(x,t)$ the unique viscosity solution to \eqref{parabolic}-\eqref{bounin1} and $v(x)$ the unique solution to \eqref{10:17}. Let us assume moreover that there exist constant $C_1, C_2$  and $\eta>0$ such that
%\[
%\left(|f(t,x)-f(x)|+|h(t,x)-h(x)|\right)d(x)^{2s-\eta}e^{\lambda t}\le C_1,
%\]
%\be\label{decayomega}
%|\Omega(t,x)\Delta\Omega(x)|d(x)^{-N-\eta}e^{\lambda t}\le C_2,
%\ee
%for some $\lambda<\overline{\lambda}$.
%Then $u(x,t)\to v(x)$ as $t\to\infty$, uniformly in $\Omega$.
%\end{teo}
%

\section{Preliminary material}

We collect in this section some preliminary lemmas, which will be
used in the proofs later on.
\subsection{Background on viscosity theory}
In the following, we will often make use of the continuity of the integral operator with respect to a suitable convergence of its variables. We state this property in its full generality, in order to be able to apply it in different contests throughout the paper.
\begin{lemma}[Continuity of the integral operator]\label{continuity}
Let us consider a sequence of points $x_n\to\bar x\in\Omega$ and a family of bounded sets $\Theta(x_n), \Theta(\bar x)$ such that $\chi_{\Theta(x_n)}\to \chi_{\Theta(\bar x)}$ almost everywhere and, for some $\delta>0$,
\[
 \ \tilde{\Theta}(x_n)\cap B_{r}(0)=\Sigma\cap B_{r}(0) \quad \mbox{ for all } r\le \delta\mbox{ and } n>0,
\]with $\Sigma$ as in \eqref{Sigmadef}. Assume moreover that $\phi_n\to\phi$ pointwise with
\be\label{c11}
\begin{split}
\left|\phi_n(x_n)-\phi_n(x_n+z)-q_n \cdot z\right|&\le C|z|^2\\
\left|\phi(\bar x)-\phi(\bar x+z)-q \cdot z\right|&\le C|z|^2
\end{split} \ \ \ \ \ \  \mbox{ for } \ \ \ |z|\le \delta,
\ee
for some $q_n, q\subset \mathbb{R}^N$  and $C$ a positive constant that does not depend on $n$. Then
\[
\mathcal{L}_{s} (\Theta(x_n),\phi_n(x_n))\to \mathcal{L}_{s} ( \Theta(\bar x),\phi(\bar x)).
\]
\end{lemma}
\begin{proof}
For any $r<\delta$, we can decompose the integral operator as follows
\be \label{eq1}
\mathcal{L}_{s} (\Theta(x_n),\phi_n(x_n))=\mathcal{L}_{s} (\Theta(x_n)\setminus B_{ r}(x_n),\phi_n(x_n))+\mathcal{L}_{s} (\Theta(x_n)\cap B_{ r}(x_n),\phi_n(x_n)).
\ee
Notice that the first integral in the right-hand side of \eqref{eq1}
is nonsingular and, for any fixed $r$, it readily  passes to
the limit as $n\to\infty$, thanks to the convergence of $\Theta(x_n)$ and $\phi_n$. Instead, the second one has to be meant in the principal value sense
\begin{align*}
\mathcal{L}_{s} (\Theta(x_n)\cap B_{ r}(x_n)&,\phi_n(x_n))\\
=\lim_{\rho\to0}\int_{\rho\le|z|\le  r}&\frac{\phi_n(x_n)-\phi_n(x_n+z)}{|z|^{N+2s}}\chi_{\Sigma}dz=\lim_{\rho\to0}\int_{\rho\le|z|\le  r}\frac{\phi_n(x_n)-\phi_n(x_n+z)-q_n\cdot z}{|z|^{N+2s}}\chi_{\Sigma}dz,
\end{align*}
where we have used that $\Sigma=-\Sigma$. Thanks to assumption \eqref{c11} we get
\[
|\mathcal{L}_{s} (\Theta(x_n)\cap B_{ r}(x_n),\phi_n(x_n))|\le C\frac{\omega_N}{2-2s}  r^{(2-2s)}.
\]
Being a similar computation valid for $\mathcal{L}_{s} ( \Theta(\bar x),\phi(\bar x))$, we get that
\[
|\mathcal{L}_{s} (\Theta(x_n),\phi_n(x_n))- \mathcal{L}_{s} ( \Theta(\bar x),\phi(\bar x))|\le |\mathcal{L}_{s} (\Theta(x_n)\setminus B_{ r},\phi_n(x_n))- \mathcal{L}_{s} ( \Theta(\bar x)\setminus B_{ r},\phi(\bar x))|+2C\frac{\omega_N}{2-2s}  r^{(2-2s)}.
\]
The assertion follows by taking the limit in the 
inequality above, %It gives us the required result to take the limit of
% the inequality above at
first as $n\to\infty$ and then as $ r\to 0$.
\end{proof}

Now we provide a suitably localized, equivalent definition of
viscosity solutions, which will turn out useful in proving the comparison
Lemma \ref{timecomp} below. Such \EEE equivalence is
already known (see, for instance, \cite{barimb}). For
completeness, we give here a statement and a proof \EEE in the 
elliptic \EEE case.

\begin{lemma}[Equivalent definition]\label{def2}
We have that \EEE $u\in \USC_b(  \Omega  )$ ($\in \LSC_b(  \Omega
)$) is a viscosity subsolution (supersolution, respectively) to the equation in  \eqref{10:17} 
if and only if, \EEE whenever $x_0 \EEE\in \Omega  $
and $\varphi\in  C^2(\Omega )$ are such that $u(  x_0 \EEE
)=\varphi( x_0 \EEE )$ and $u( y)\le \varphi( y)$ for all $y \in \Omega$, then for all $B_r( x_0 \EEE )\subset \Omega$ the function
\begin{equation}\label{16:57}
 \varphi_r(x)=
\begin{cases}
\varphi(x) \qquad & \mbox{in } B_r( x_0  ),\\
 u(x) & \mbox{otherwise, }
\end{cases}
\end{equation}
satisfies
\begin{equation}\label{serve}
  h( x_0 \EEE ) \varphi_r(  x_0 \EEE )+\mathcal{L}_{s} (\Omega( x_0 \EEE),\varphi_r(  x_0 \EEE
   ))\le (\ge) \ f(   x_0 \EEE ). \EEE
\end{equation}
 A similar result hold in the parabolic case. 
%Parabolic case: \EEE
%We have that $u\in \USC_b( (0,T)\times\Omega  )$ ($\in \LSC_b( (0,T)\times\Omega  )$) is a viscosity subsolution (supersolution, respectively) to the equation in  \eqref{parabolic}
%if and only if, whenever $( t_0 , x_0 )\in (0,T)\times\Omega  $ and $\varphi\in
%C^2((0,T)\times{\Omega})$ are such that $u( t_0 , x_0 )=\varphi( t_0 , x_0 )$
%and $u(\tau,y)\le \varphi(\tau,y)$ for all $\tau \EEE,y\in(0,T)\times\Omega$, then for all $B_r( x_0 \EEE )\subset \Omega$ the function
%\begin{equation}\label{16:57}
% \varphi_r=
%\begin{cases}
%\varphi \qquad & \mbox{in } B_r( x_0  )\times(0,T),\\
% u & \mbox{otherwise, }
%\end{cases}
%\end{equation}
%satisfies 
%\begin{equation*}
% \partial_t  \varphi( t , x )+ h( x ) \varphi( t , x )+\mathcal{L}_{s} (\Omega(t,x),\varphi( t , x ))\le (\ge) \ f( t , x ). 
%\end{equation*}
% Moreover $u\in C([0,T)\times\overline \Omega)$ is a \emph{viscosity solution} to \eqref{parabolic} if it is both sub- and supersolution and satisfies the boundary and initial conditions
%  \begin{align}
%    \label{parabolic}
% \left\{   \begin{array}{ll}
%    %\label{bounin1}
% u(t,x) = 0 & \mbox{on } (0,T)\times\partial\Omega\, ,\\
%    u(0,x) = u_0(x) & \mbox{on }   \Omega,
%           \end{array}
%                      \right.
%\end{align}
%pointwise.\EEE
\end{lemma}
\begin{proof}
We prove only the equivalence in the case of  subsolutions,  
the case of supersolution being identical. \EEE

Let us assume at first that $u\in \USC_b(\Omega)$ fulfills the
conditions of Lemma \ref{def2}. We want to check that is a
viscosity subsolution in the sense of Definition \ref{def1}. Let
$\varphi \EEE \in  C^2(\Omega)$ be such that
$u-\varphi $ has a global maximum at $x_0$ and $\varphi(x_0)=u(x_0)$.
It then follows that, for any  $B_{r}(x_0)\subset\Omega$,
\begin{align*}
f(x_0)&\ge  h(x_0)\varphi_r( x_0 )+\mathcal{L}_{s} (\Omega(x_0),\varphi_r( x_0 ))\\[1mm]
 &=h(x_0)u(x_0)+\int_{ \Omega(x)\setminus B_{r}(x_0) }  \frac{u(x)-u(y)}{|x-y|^{N+2s}} dy+\int_{  \Omega(x)\cap B_{r}(x_0) }  \frac{u(x)-\varphi(y)}{|x-y|^{N+2s}} dy\\
&\ge  h(x_0)u(x_0)+\int_{  \Omega(x)}
                                                                                                                                                            \frac{u(x)-\varphi(y)}{|x-y|^{N+2s}} dy\\[1mm]
  &=h(x_0)\varphi ( x_0 )+\mathcal{L}_{s} (\Omega(x_0),\varphi (
    x_0 )) , 
\end{align*}
where the first inequality $\ge$ comes from
\eqref{serve} and the second one follows as $u\le \varphi$ in $\Omega$.

To show the reverse implication, let us assume that
$u\in \USC_b(\Omega)$ \EEE is a viscosity subsolution to \eqref{10:17} according to Definition \ref{def1}. % We want to show that $u$ is also a viscosity subsolution to \eqref{10:17} according to Definition \ref{def2}. 
Let $\varphi\in  C^2({\Omega})$ be such that $u-\varphi$ has a maximum
at $x_0\in\Omega$ and $\varphi(x_0)=u(x_0)$ and, for any
$B_{r}(x_0)\subset\Omega$, let $\varphi_{r}$ be the auxiliary function
defined in \eqref{16:57}.  As a first step, \EEE we modify
$\varphi_{r}$ \EEE  as $\varphi_{r,n}(x)=\varphi_{r}(x)
\EEE+\frac1n|x-x_0|^2$ and notice that, for any $n\in\mathbb{N}$,
the function \EEE $u-\varphi_n$ has a strict local maximum at $x_0$ and 
\[
u-\varphi_{r,n}\le -\frac  {r^2} {4n} \ \ \ \mbox{in} \ \ \ \Omega\setminus B_{\frac{ r}{2}}(x_0).
\]

Let $\psi_1,\psi_2\in C^{\infty}(\Omega)$ be  a partition of unity
\EEE
%Then, let us consider the the  partition of unity $\psi_1,\psi_2\in
%C^{\infty}(\Omega)$
associated to the concentric balls $B_{\frac{ r}{2}}(x_0)$ and
$B_{\frac{3 r}{4}}(x_0)$, namely, $0\le\psi_ , \, \psi_2\le 1$, $\psi_1=1$ in $B_{\frac{ r}{2}}(x_0)$, $\psi_1=0$ in $\Omega\setminus B_{\frac{3 r}{4}}(x_0)$, and $\psi_1+\psi_2=1$. Let us finally set 
\[
\zeta_n=\psi_1\varphi_{r,n}+\rho_{m_n}* (\psi_2\varphi_{r,n})
\]
where $\rho_{m_n}$ is a mollifier and $\{m_n\}_{n\in\mathbb{N}}$ is a
sequence of numbers converging to $0$ to be suitably determined. Notice that $\zeta_n(x)\equiv \varphi_{r}(x)\EEE+\frac1n|x-x_0|^2$ in
$B_{\frac{ r}{2}}(x_0)$ and that $\zeta_n(x_0)=u(x_0)$ for $n$
large.  Moreover, thanks to the properties of mollifiers, for any $n\in \mathbb N$ there exists $m_n\in\mathbb{N}$ such that
\[
u-\zeta_n\le -\frac{ r^2}{8n} \ \ \ \mbox{in} \ \ \ \Omega\setminus B_{\frac{ r}{2}}(x_0).
\] 
Then, $\zeta_n\in C^{2}$ is a good test function for Definition
\ref{def1} and we have
\[
h(x_0)u(x_0)+\mathcal{L}_{s}(\Omega(x_0),\zeta_n(x_0)) \le \ f(x_0). 
\]
Taking the limit as $n\to\infty$ we prove the assertion applying Lemma \ref{continuity}. Note
indeed that $\zeta_n\to \varphi$ in $C_2(B_{\frac r2}(x_0))$ and $\zeta_n\to u$ pointwise in $\Omega\setminus B_{\frac r2}(x_0)$.
\end{proof}
In proving the stability of  families of viscosity solutions, a
suitable notion of limit for sequences of upper semicontinuous
functions has to be considered, see for instance \cite{caff}. We
introduce the following.
% To this
% aim 
% Now we recall the notion of $\Gamma-$convergence for uppersemicontinuous functions, see for instance \cite{caff}. It represents the minimal requirement to prove stability for families of viscosity solutions.
%\begin{defin}[$\Gamma$-convergence]
%A sequence of upper-semicontinuous functions $v_n$ is said to
%\emph{$\Gamma$-converge to $v$ at $\bar z\in D\subset
%\mathbb{R}^M$}  if
%\begin{align}
%  &\label{gamma1}
%\forall   z_n\to \bar z: \quad 
%    \limsup_{n\to\infty}v_n(z_n)\le v(\bar z)\\
%  &\label{gamma2}
% \exists   z_n\to \bar z \ \ : \ \
%    \lim_{n\to\infty}v_n(z_n)=v(\bar z).
%    \end{align}
%\end{defin} 
\begin{defin}[$\Gamma$-convergence]
A sequence of upper-semicontinuous functions $v_n$ is said to
\emph{$\Gamma$-converge to $v$ in $ D\subset
\mathbb{R}^M$}  if
\begin{align}
  &\label{gamma1}
\mbox{for all converging sequences }   z_n\to \bar z \mbox{ in } D \ : \quad 
    \limsup_{n\to\infty}v_n(z_n)\le v(\bar z)\\
  &\label{gamma2}
 \mbox{for all $\bar z\in D$ there exists a sequence }   z_n\to \bar z \ \ : \ \
    \lim_{n\to\infty}v_n(z_n)=v(\bar z).
    \end{align}
\end{defin} 
  This concept corresponds (up to a sign change) to a localized version of the
  classical $\Gamma$-convergence notion, see  \cite{DalMaso93}, hence
  the same name.
  
Clearly, uniformly converging sequences in $\Omega$ are also
$\Gamma$-converging.
Moreover,
$\Gamma$-convergence readily ensues in connection with the
upper-semicontinuous envelope  of a family of functions. Both
examples  will play a role in the sequel.
%This type of convergence is deeply linked with the notion of
%viscosity solution, see for instance Proposition 4.3 in \cite{CIL}
%and Lemma 4.5 in \cite{caff}. The notion of $\Gamma-$convergence
%represents the minimum requirement to have stability  The two
%ingredients above represent the minimum requirements to show that a
%certain family of subsol have some stability properties for some
%family of viscosity (sub)solutions.

The following stability result is an adaptation of the classical one provided in Proposition 4.3 of \cite{CIL} (see also Lemma 4.5 in \cite{caff}).
\begin{lemma}[Stability]\label{stability}
Let us consider $v\in \USC_b((0,T)\times\Omega)$ and $f\in
C((0,T)\times\Omega)$. Assume moreover that
\begin{align*}
i)&\quad \{v_n\}\subset \USC_b((0,T)\times\Omega)  \ \ \text{$\Gamma-$converges to $v$  in $(0,T)\times\Omega$,}\\
 ii)&\quad   f_n \to f, h_n \to h \ \ \text{locally uniformly}, \mbox{ and } |\Omega_n\triangle\Omega|\to \mbox{ as } n\to\infty, %\ \chi_{\Omega_n}\to \chi_{\Omega} \ \ \text{almost everywhere in} \ \Omega,
\\
iii)&\quad  
\partial_tv_n(t,x) +h_n( t, x)v_n(t,x) +\mathcal{L}_{s} (\Omega_n( t, x),v_n( t,
      x)  )\le  f_n(
      t, x) \ \ \ \mbox{in} \ \ \ (0,T)\times\Omega \ \ \text{ 
  in the viscosity sense.}
\end{align*}
  Then, if $\Omega( t, x)$ satisfies \eqref{contpar},\eqref{omegax}, it follows that
\[
\partial_tv( t, x)+h( t, x)v( t, x)+\mathcal{L}_{s} (\Omega( t, x),v( t, x))\le  f( t, x)
\]
in the viscosity sense.
\end{lemma}
\begin{remark}\label{ellipticparabolic}\rm
An elliptic version of the Lemma holds true by assuming \eqref{continuita}-\eqref{Sigmadef}. Notice that the stationary case  can be straightforwardly obtained from the evolutionary one upon interpreting  %The proof of Corollary
                                %\ref{ellipticdif} follows easily from
                                %the following facts: any
  $u:\Omega \to \mathbb{R}$  as a trivial time-dependent function
  $\tilde u(t,x)= u(x)$  on   $(0,T)\times\Omega$. In
  fact, if  such a  function is touched from above or from below
  by a smooth function $\phi\in C^2((0,T)\times\Omega)$ at some
  $(t,x)\in(0,T)\times\Omega$,  we have that   $\partial_t
  \varphi(t,x)=0$. We hence conclude that  such   time-dependent
  representation $\tilde u(t,x) $  of a subsolution (or supersolution) $u(x)$ of the
  elliptic problem is subsolution (supersolution, respectively) of
  its  parabolic counterpart. 
\end{remark}
\begin{proof}[ Proof of Lemma \ref{stability}]
Let us assume that $v-\varphi$ has a strict global maximum, equal to
$0$, at $(\bar t,\bar x)\in (0,T)\times \Omega$. Taking $\varphi_{\theta}=\varphi+\theta (|t-\bar t|^2+|x-\bar x|^2)$, we have that also the $\sup (v-\phi_{\theta})$ is reached only at $(\bar t,\bar x)$.
Owing to the assumption $i)$,  we know that there exists a sequence of points $\{( \tau_{n}, y_{n})\}\subset (0,T)\times\Omega$ such that
\be\label{6.10}
( \tau_{n}, y_{n}, v_n(\tau_{n}, y_{n}))\to (\bar t,\bar x, v(\bar t,\bar x)).
\ee
%Let us take $\delta$ small enough to have $B_{\delta}(\bar t,\bar
%x):=\{(t,x)\ : \   (|t-\bar t|^2+ |x-\bar
%x|^2)^{\frac12}< \delta\} \subset  (0,T)\times\Omega$ and $n$ large enough so that $( \tau_{n}, y_{n})\in B_{\delta}(\bar t,\bar x)$.
Thanks to the penalization in the definition of $\varphi_{\theta}$ and assumption $i)$, for $n$ large enough we have that 
\[
v_n(\tau_{n}, y_{n})-\varphi_{\theta}(\tau_{n}, y_{n})\le \sup_{(0,T)\times \Omega} (v_n-\varphi_{\theta})=v_n(t_{n}, x_{n})-\varphi_{\theta}(t_{n}, x_{n}) =\epsilon_n,
\]
for some  $\{(t_{n}, x_{n})\}\in (0,T)$ such that, up to a not relabeled subsequence, $(t_n,x_n)\to(\tilde t, \tilde x)\in (0,T)\times\Omega$. Using again the $\Gamma-$convergence of $v_n$, we find 
\[
(v-\varphi_{\theta})(\tilde t, \tilde x)\ge \limsup_{n\to\infty }(v_n-\varphi_{\theta})(t_{n}, x_{n}) \ge\lim_{n\to\infty}(v_n-\varphi_{\theta})(\tau_{n}, y_{n})=(v-\varphi_{\theta})(\bar t,\bar x)=0.
\]
Since the supremum of $v-\varphi$ is strict, this implies that $(\tilde t, \tilde x)=(\bar t, \bar x)$ and that $\epsilon_n\to 0$.
Moreover, setting $\varphi_n=\varphi_{\theta}+\epsilon_n$, it follows that $\sup_{(0,T)\times \Omega}(v_n-\varphi_{n})=(v_n-\varphi_{\theta})(t_{n}, x_{n})=0$.

In conclusion,  we have proved that $v_n-\varphi_n$ has a global maximum at $(t_n,x_n)$ (for $n$ large enough) and $v_n(t_n,x_n)=\varphi_n(t_n,x_n)$, that $\epsilon_n\to 0$, and that $(t_n,x_n)\to(\bar t, \bar x)$. 
Since each $v_n$ is a subsolution at $(\bar t, \bar x)$  we get
\[
\partial_t\varphi_n( t_n,x_n)+h_n( t_n,x_n)v_{n}( t_n,x_n)+\mathcal{L}_{s} (\Omega_n(t_n,x_n),\varphi_n( t_n,x_n))\le f_n( t_n,x_n).
\]
The first two terms in the left-hand side and the one in the right-hand
side  easily pass to the limit  for $n \to \infty$, 
thanks to the continuity of $h$ and  $f$ and to  the
definition of $\varphi_n$. In order  to deal with the
integral operator notice that
\[
\mathcal{L}_{s} (\Omega_n(t_n,x_n),\varphi_n( t_n,x_n))=\mathcal{L}_{s} (\Omega_n(t_n,x_n),\varphi( t_n,x_n)).
\]
Since $\varphi$ is smooth, we can use Lemma \ref{continuity}, with
$\Theta(x_n)=\Omega_n(t_n,x_n)$ and $\phi_n(\cdot)=\varphi(
t_n,\cdot)$,  and  pass to the limit with respect to $n$.
Eventually,  we get
\[
\partial_t\varphi_{\theta}( \bar t,\bar x)+h( \bar t,\bar x)v ( \bar t,\bar x)+\mathcal{L}_{s} (\Omega( \bar t,\bar x),\varphi_{\theta}( \bar t,\bar x))\le f( \bar t,\bar x),
\]
for any $\theta>0$. Taking the limit (using Lemma \ref{continuity} again) as $\theta\to 0$, we obtain the desired result.\EEE
\end{proof}

The previous stability result highlights the robustness of 
% how much robust
the notion of viscosity solution in relation to limit 
procedures. Notice that, given any uniformly bounded sequence of
viscosity (sub/super) solutions of a certain family of equations, 
one can always find a $\Gamma-$limit  and this is  the candidate (sub/super) solution for the limiting equation. Such candidate is given by the lower/upper half relaxed limit 
\[
\overline{v}(x)=\sup\{\limsup_{n\to\infty} v_n(x_n) \ : \ x_n\to x\}, \ \ \ \ \underline{v}(x)=\inf\{\liminf_{n\to\infty} v_n(x_n) \ : \ x_n\to x\}.
\]
It is easy to check that $v_n$ $\Gamma-$converges to $\overline{v}$
and $-v_n$ $\Gamma-$converges to $-\underline{v}$. The key point here
is that no compactness on the sequence $v_n$ is required for the
existence of $\overline{v}$ and  $\underline{v}$, as  boundedness
suffices.  As we shall see in Section \ref{seceigenvalue}, this is a particularly powerful tool when dealing with equations that satisfy a comparison principle.
\EEE

\subsection{Sup- and infconvolution}
In the sequel we often need to determine the equation (or inequality)
solved by the difference of sub- or supersolutions. Note that,
since such functions are not smooth, the  property of being a
sub- or supersolution may be not preserved by taking
differences.  To deal with this difficulty, we need to use
suitable regularization of the involved functions.  Let us start
recalling the definition of {\it supconvolution} of $u\in
\USC_b((0,T)\times\Omega)$, namely, 
\be\label{supcon}
u^{\epsilon}(t,x)=\sup_{(\tau, y) \in (0,T)\times\Omega}\left\{u(\tau,y)-\frac1{\epsilon}(|x-y|^2+|t-\tau|^2)\right\}.
\ee
Notice that, since $u$ is upper semicontinuous and bounded, for
$\epsilon$ small enough the supremum above is reached inside
$(0,T)\times\Omega$. To be more precise, let us adopt the following
notation: for any $(t,x)\in(0,T)\times\Omega$,  let 
$(t^{\epsilon},x^{\epsilon})$  be  a point with the following property
\be\label{01-06bis}
u^{\epsilon}(t,x)=u(t^{\epsilon},x^{\epsilon})-\frac1{\epsilon}(|x-x^{\epsilon}|^2+|t-t^{\epsilon}|^2).
\ee
Then,
\be\label{control}
(|x-x^{\epsilon}|^2+|t-t^{\epsilon}|^2)\le 2\epsilon \|u\|_{L^{\infty}((0,T)\times\Omega)}.
\ee
Moreover, by construction, it results that the parabola
\[
P(t,x)=u(\bar{t}^{\epsilon},\bar{x}^{\epsilon})-\frac1{\epsilon}(|t-\bar{t}^{\epsilon}|^2+|x-\bar{x}^{\epsilon}|^2)
\]
touches $u^{\epsilon}$ from below at $(\bar t, \bar x)$. This
shows that the supconvolution is  semiconvex in
$(0,T)\times\Omega$. Such a property is particularity useful in order
to pointwise evaluate some viscosity inequality related to
subsolutions. Indeed, let us assume that $u^{\epsilon}$ is touched
from above by a smooth function at $(\bar t, \bar x)$. Then, thanks to
its semiconvexity property, we deduce that $u^{\epsilon}\in
C^{1,1}(\bar t, \bar x)$, namely there exist $q  \in\mathbb{R}^{N+1}$ and $C>0$ such that, in a neighborhood of $(\bar t, \bar x)$,
\be\label{c11}
\left|u^{\epsilon}(t,x)-u^{\epsilon}(\bar t, \bar x)-q \cdot\binom{t-\bar t}{x-\bar x}\right|\le C(|t-\bar{t}|^2+|x-\bar{x}|^2).
\ee
This means that the time derivative and the fractional operator can be
evaluated pointwise at $(\bar t, \bar x)$ (see Lemma \ref{pointsub}
below).

Similarly, the {\it infconvolution} of a function $v\in \LSC((0,T)\times\Omega)\cap L^{\infty}((0,T)\times\Omega)$ is defined as
\be\label{infconv}
u_{\epsilon}(t,x)=\inf_{(\tau, y) \in (0,T)\times\Omega}\left\{u(\tau,y)+\frac1{\epsilon}(|x-y|^2+|t-\tau|^2)\right\},
\ee
and we let $(t_{\epsilon},x_{\epsilon})$ be the point
where 
\be\label{01-06tris}
u_{\epsilon}(t,x)=u(t_{\epsilon},x_{\epsilon})-\frac1{\epsilon}(|x-x_{\epsilon}|^2+|t-t_{\epsilon}|^2).
\ee
The property of the infconvolution correspond  to those of
the supconvolution up to the trivial transformation
$v_{\epsilon}=-(-v)^{\epsilon}$. Omitting further details for the
sake of brevity, we limit ourselves in proving the following inequality
on supconvolutions.  

\begin{lemma}\label{pointsub}
Let us assume \eqref{Sigmadef}, \eqref{contpar}, and \eqref{omegax} and that $u(t,x)$ is a viscosity subsolution to
\eqref{parabolic} and let $u^{\epsilon}(t,x)$ be its
supconvolution. If  $u^{\epsilon}$ is touched from above by some smooth function at $(\bar t, \bar x)$, the following inequality holds in a classical sense
\be\label{13:34uno}
\partial_tu^{\epsilon}(\bar t, \bar x)+h(\bar{t}^{\epsilon},\bar{x}^{\epsilon})u^{\epsilon}(\bar t, \bar x)+\mathcal{L}_{s} (\Omega(\bar{t}^{\epsilon},\bar{x}^{\epsilon}),u^{\epsilon}(\bar t, \bar x))\le f(\bar{t}^{\epsilon},\bar{x}^{\epsilon}),
\ee
where $(\bar{t}^{\epsilon},\bar{x}^{\epsilon})$ satisfies \eqref{01-06bis}.
\end{lemma}
\begin{proof} Let us assume that $u^{\epsilon}$ is touched from above by a smooth $\varphi$ at $(\bar t,\bar x)$. Let us recall that there exists $(\bar{t}^{\epsilon},\bar{x}^{\epsilon})$ such that
\[
u^{\epsilon}(\bar t,\bar x)=u^{\epsilon}(\bar{t}^{\epsilon},\bar{x}^{\epsilon})-\frac1{\epsilon}(|\bar{t}-\bar{t}^{\epsilon}|^2+|\bar{x}-\bar{x}^{\epsilon}|^2) \ \ \ \mbox{and} \ \ \ (\bar{t}^{\epsilon},\bar{x}^{\epsilon})\to(\bar t,\bar x) \ \ \ \mbox{as} \ \ \ \epsilon\to0.
\]
By definition of supconvolution we have that
\[
u^{\epsilon}(t+\bar t-\bar{t}^{\epsilon},x+\bar x-\bar{x}^{\epsilon})\ge u(\tau,y)-\frac1{\epsilon}(|t+\bar t-\bar{t}^{\epsilon}-\tau|^2+|x+\bar x-\bar{x}^{\epsilon}-y|^2).
\]
Choosing $(\tau,y)=(t,x)$ it follows
\[
u^{\epsilon}(t+\bar t-\bar{t}^{\epsilon},x+\bar x-\bar{x}^{\epsilon})\ge u(t,x)-\frac1{\epsilon}(|\bar t-\bar{t}^{\epsilon}|^2+|\bar x-\bar{x}^{\epsilon}|^2).
\]
Then, by  defining 
\[
\bar{\varphi}(t,x)=\varphi(t+\bar t-\bar{t}^{\epsilon},x+\bar x-\bar{x}^{\epsilon})+\frac1{\epsilon}(|\bar t-\bar{t}^{\epsilon}|^2+|\bar x-\bar{x}^{\epsilon}|^2),
\]
we infer that $\bar{\varphi}$ touches from above $u$ at $(\bar{t}^{\epsilon},\bar{x}^{\epsilon})$. Since $u$ is a viscosity subsolution to \eqref{parabolic}, it follows that
\be\label{01-06}
\partial_t\bar{\varphi}_{r}(\bar{t}^{\epsilon},\bar{x}^{\epsilon})+h(\bar{t}^{\epsilon},\bar{x}^{\epsilon})\bar{\varphi}_{r}(\bar{t}^{\epsilon},\bar{x}^{\epsilon})+\mathcal{L}_{s} (\Omega(\bar{t}^{\epsilon},\bar{x}^{\epsilon}),\bar{\varphi}_{r}(\bar{t}^{\epsilon},\bar{x}^{\epsilon}))\le f(t,x).
\ee
Now notice that by construction of $\bar{\varphi}$ it results $\partial_t\bar{\varphi}_{r}(\bar{t}^{\epsilon},\bar{x}^{\epsilon})=\partial_t{\varphi}_{r}(\bar{t},\bar{x})$ and
\[
 \bar{\varphi}_{r}(\bar{t}^{\epsilon},\bar{x}^{\epsilon})-\bar{\varphi}_{r}(\bar{t}^{\epsilon},\bar{x}^{\epsilon}+z)= \varphi_{r}(\bar t,\bar x)-\varphi_{r}(\bar t,\bar x+z),
\]
that implies
\[
\mathcal{L}_{s} (\Omega(\bar{t}^{\epsilon},\bar{x}^{\epsilon}),\bar{\varphi}_{r}(\bar{t}^{\epsilon},\bar{x}^{\epsilon}))=\int_{\Omegat(\bar{t}^{\epsilon},\bar{x}^{\epsilon})}\frac{\bar{\varphi}_{r}(\bar{t}^{\epsilon},\bar{x}^{\epsilon})-\bar{\varphi}_{r}(\bar{t}^{\epsilon},\bar{x}^{\epsilon}+z)}{|z|^{N+2s}}dz
\]
\[
=\int_{\Omegat(\bar{t}^{\epsilon},\bar{x}^{\epsilon})}\frac{\varphi_{r}(\bar t,\bar x)-\varphi_{r}(\bar t,\bar x+z)}{|z|^{N+2s}}dz=\mathcal{L}_{s} (\Omega(\bar{t}^{\epsilon},\bar{x}^{\epsilon}),{\varphi}_{r}(\bar{t},\bar{x})).
\]
Then, \eqref{01-06} becomes
\[
\partial_t{\varphi}_{r}(\bar{t},\bar{x})+h(\bar{t}^{\epsilon},\bar{x}^{\epsilon})u^{\epsilon}(\bar{t},\bar{x})+\mathcal{L}_{s} (\Omega(\bar{t}^{\epsilon},\bar{x}^{\epsilon}),{\varphi}_{r}(\bar{t},\bar{x}))\le f(\bar{t}^{\epsilon},\bar{x}^{\epsilon}).
\]
Since $u^{\epsilon}$ is touched from above by a smooth
function at $(\bar t,\bar x)$, we know that $u^{\epsilon}\in
C^{1,1}(\bar t,\bar x)$ (see \eqref{c11}) and then
$\partial_t{\varphi}_{r}(\bar{t},\bar{x})=\partial_tu^{\epsilon}(\bar{t},\bar{x})$. Recalling assumption \eqref{omegax} and since $(\bar{t}^{\epsilon},\bar{x}^{\epsilon})\to (\bar t,\bar x)\in (0,T)\times \Omega$, we deduce that there exists $\delta>0$ such that, for any $r<\delta$, we can decompose the nonlocal operator as follows
\begin{align*}
\mathcal{L}_{s}
  (\Omega(\bar{t}^{\epsilon},\bar{x}^{\epsilon}),{\varphi}_{r}(\bar{t},\bar{x}))
  &=\int_{\Sigma \cap B_r(0)}\frac{\varphi(\bar t,\bar x)-\varphi(\bar
    t,\bar x+z)}{|z|^{N+2s}}dz\\
  &\quad +\int_{\Omegat(\bar{t}^{\epsilon},\bar{x}^{\epsilon})\setminus B_r(0)}\frac{u^{\epsilon}(\bar t,\bar x)-u^{\epsilon}(\bar t,\bar x+z)}{|z|^{N+2s}}dz.
\end{align*}
The integral on $\Sigma\cap B_r(0))$ is well defined converges to zero as $r\to0$, due to the smoothness of $\varphi$ and the symmetry of $\Sigma$. To deal with the second integral we apply Lemma \ref{continuity} with $x_n=\bar x^{\epsilon}$,  $\Theta(x_n)=\Omega(\bar{t}^{\epsilon},x_n)\setminus B_{\frac1n}(x_n))$ and $\phi_{n}(\cdot)=\phi(\cdot)=u^{\epsilon}(\bar t,\cdot)$. We deduce that $\mathcal{L}_{s}(\Omega(\bar{t}^{\epsilon},\bar{x}^{\epsilon}),{\varphi}_{r}(\bar{t},\bar{x}))\to
\mathcal{L}_{s}
(\Omega(\bar{t}^{\epsilon},\bar{x}^{\epsilon}),u^{\epsilon}(\bar{t},\bar{x}))$
as $r\to0$. This completes the proof of the Lemma.
\end{proof}
A similar inequality holds for infconvolution. For the sake of
later reference, we state it here below without proof. This  can be
obtained by straightforwardly adapting the argument of Lemma
\ref{pointsub}. 
\begin{lemma}\label{pointsup}
Let us assume that $v(t,x)$ is a viscosity supersolution to
\eqref{parabolic} and let $v_{\epsilon}(t,x)$ be its
infconvolution. If $v_{\epsilon}$ is touched from below by some smooth function at $(\bar t, \bar x)$, the following inequality holds in a classical sense
\be\label{13:34due}
\partial_tv_{\epsilon}(\bar t, \bar x)+h(\bar{t}_{\epsilon},\bar{x}_{\epsilon})v_{\epsilon}(\bar t, \bar x)+\mathcal{L}_{s} (\Omega(\bar{t}_{\epsilon},\bar{x}_{\epsilon}),v_{\epsilon}(\bar t, \bar x))\le f(\bar{t}_{\epsilon},\bar{x}_{\epsilon}),
\ee
where $(t_{\epsilon},x_{\epsilon})$ satisfies \eqref{01-06tris}.
\end{lemma}
In the following Lemma we  eventually state that the difference of
a super- and a subsolution is still a supersolution.   
\begin{lemma}[Difference]\label{differenza}
Let us consider $h_1(t,x), h_2(t,x)$ that satisfy \eqref{alphap}, $\{\Omega_1(t,x)\}, \{\Omega_2(t,x)\}$ that satisfy \eqref{contpar},\eqref{omegax} and two functions $u\in USC_b((0,T)\times\Omega), v\in LSC_b((0,T)\times\Omega)$ that solve in the viscosity sense
\[
\partial_tu(t,x)+h_1(t,x)u(t,x)+\mathcal{L}_{s} (\Omega_1(t,x),u(t,x))\le f_1(t,x) \qquad \mbox{in }  (0,T)\times\Omega
\]
\[
\partial_tv(t,x)+h_2(t,x)v(t,x)+\mathcal{L}_{s} (\Omega_2(t,x),v(t,x))\ge f_2(t,x) \qquad \mbox{in }  (0,T)\times\Omega,
\]
respectively. Then $w=u-v$ solves in the viscosity sense
\[
 \partial_tw(t,x)+h_1(t,x)w(t,x)+\mathcal{L}_{s} (\Omega_1(t,x),w(t,x))\le \tilde{f}(x,t) \qquad \mbox{in }  (0,\infty)\times\Omega
\]
where
\[
\tilde{f}(x,t)=f_1({t} ,{x} )-f_2({t} ,{x} )+M|h_1({t} ,{x} )-h_2({t} ,{x} )|+ 2M\int_{ |z|\ge \frac{\zeta}{2}d(\bar x) }\frac{|\chi_{\Omegat_1({t} ,{x} )}-\chi_{\Omegat_2({t} ,{x} )}|}{|z|^{N+2s}}dz,\EEE
\]
\end{lemma}
with $M=\max\{\|u\|_{L^{\infty}((0,T)\times\Omega)},\|v\|_{L^{\infty}((0,T)\times\Omega)}\}$.
\begin{proof}
Recalling definitions \eqref{supcon} and \eqref{infconv}, let us
consider the function
$w^{\epsilon}(t,x)=u^{\epsilon}(t,x)-v_{\epsilon}(t,x)$ and assume
that it is touched from above by a $\varphi\in
C^2((0,\infty)\times\Omega)$ at point $ (\bar{t}, \bar{x})$. This 
means  that
\[
u^{\epsilon}(\bar{t}, \bar{x})-v_{\epsilon}(\bar{t},
\bar{x})=\varphi(\bar{t}, \bar{x}) \ \ \ \mbox{and} \ \ \ 
u^{\epsilon}-v_{\epsilon}\le \varphi  %\ \ \ u^{\epsilon}\le
                                %v_{\epsilon}+\varphi \ \ \
                                %u^{\epsilon}-\varphi \le v_{\epsilon}
\ \ \mbox{in} \ \ \Omega.
\]

 This latter fact, together with the semiconvexity property of both
 $u^{\epsilon}$ and $-v_{\epsilon}$, implies that $u^{\epsilon}$ and
 $-v_{\epsilon}$ are $C^{1,1}(\bar{t},\bar{x})$ (see
 \eqref{c11}). We are hence in the position of applying
 Lemmas \ref{pointsub} and \ref{pointsup} and evaluating the
 inequalities satisfied by $u^{\epsilon}$ and $v_{\epsilon}$ pointwise. We have that
\[
\partial_tu^{\epsilon} (\bar{t}, \bar{x})+h_1(\bar{t}^{\epsilon},\bar{x}^{\epsilon})u^{\epsilon} (\bar{t}, \bar{x})+\mathcal{L}_{s}(\Omega_1(\bar{t}^{\epsilon},\bar{x}^{\epsilon}),u^{\epsilon} (\bar{t}, \bar{x}))\le f_1(\bar{t}^{\epsilon},\bar{x}^{\epsilon}),
\]
and that
\[
\partial_tv_{\epsilon} (\bar{t},
\bar{x})+h_2(\bar{t}_{\epsilon},\bar{x}_{\epsilon})v_{\epsilon}
(\bar{t}, \bar{x})+\mathcal{L}_{s}(\Omega_2(
\bar{t}_{\epsilon},\bar{x}_{\epsilon}), v_{\epsilon} (\bar{t}, \bar{x}))\ge f_2(\bar{t}_{\epsilon},\bar{x}_{\epsilon}).
\]
Recalling the ordering assumption between $w^{\epsilon}$ and $\varphi$ and combining the two inequalities above, we infer that, for $\epsilon$ small enough,
\begin{align}
  \label{01-06tris}
&\partial_t \varphi_r(\bar t,\bar
  x)+h_1(\bar{t}^{\epsilon},\bar{x}^{\epsilon})\varphi_r(\bar t,\bar
  x)+ \mathcal{L}_s 
  (\Omega_1(\bar{t}^{\epsilon},\bar{x}^{\epsilon}),  \varphi_r(\bar
  t,\bar  x))\\[2mm]
  &\nonumber \quad \le
\partial_t w^{\epsilon}(\bar t,\bar
    x)+h_1(\bar{t}^{\epsilon},\bar{x}^{\epsilon})w^{\epsilon}(\bar t,\bar
    x)+\mathcal{L}_s(\Omega_1(\bar{t}^{\epsilon},\bar{x}^{\epsilon}),
    w^{\epsilon}(\bar t,\bar  x))\\[2mm]
  &\nonumber \quad 
\le
    f_1(\bar{t}^{\epsilon},\bar{x}^{\epsilon})-f_2(\bar{t}_{\epsilon},\bar{x}_{\epsilon})+M|h_1(\bar{t}^{\epsilon},\bar{x}^{\epsilon})-h_2(\bar{t}_{\epsilon},\bar{x}_{\epsilon})|\\
  &\nonumber \qquad +2M\int_{\mathbb{R}^N}\frac{|\chi_{\Omegat_1(\bar{t}^{\epsilon},\bar{x}^{\epsilon})}-\chi_{\Omegat_2(\bar{t}_{\epsilon},\bar{x}_{\epsilon})}|}{|z|^{N+2s}}dz.
\end{align}
Let us stress that, for $\epsilon$ small enough, the integral term in the right hand side above is finite. Indeed, thanks to the assumption \eqref{omegax} and since $(\bar{t}^{\epsilon},\bar{x}^{\epsilon})\to (\bar{t},\bar{x})$ and $(\bar{t}_{\epsilon},\bar{x}_{\epsilon})\to (\bar{t},\bar{x})$  as $\epsilon\to 0$, then 
\[
\exists \ \epsilon_0>0 \ : \ \forall \epsilon\in(0,\epsilon_0) \ \ \ B_{\frac{\zeta}{2}d(\bar x)}\cap \Omegat(\bar{t}^{\epsilon},\bar{x}^{\epsilon})= B_{\frac{\zeta}{2}d(\bar x)}\cap \Omegat(\bar{t}_{\epsilon},\bar{x}_{\epsilon})=B_{\frac{\zeta}{2}d(\bar x)}\cap \Sigma.
\]
This implies that
\begin{align*}
&\int_{\mathbb{R}^N}\frac{|\chi_{\Omegat_1(\bar{t}^{\epsilon},\bar{x}^{\epsilon})}-\chi_{\Omegat_2(\bar{t}_{\epsilon},\bar{x}_{\epsilon})}|}{|z|^{N+2s}}dz=2M\int_{|z|\ge
                 \frac{\zeta}{2}d(\bar
                 x)}\frac{|\chi_{\Omegat_1(\bar{t}^{\epsilon},\bar{x}^{\epsilon})}-\chi_{\Omegat_2(\bar{t}_{\epsilon},\bar{x}_{\epsilon})}|}{|z|^{N+2s}}dz\\
  &\qquad \le  \left(\frac{2}{\zeta d(\bar x)}\right)^{N+2s}|\Omegat_1(\bar{t}^{\epsilon},\bar{x}^{\epsilon})\triangle\Omegat_2(\bar{t}_{\epsilon},\bar{x}_{\epsilon})|.
\end{align*}
Since \eqref{01-06tris} is true any time that $w^{\epsilon}$ is touched from above by a smooth $\varphi$ at some point in $(0,T)\times\Omega$, we can conclude that $w^{\epsilon}$ solves in the viscosity sense
\[
\partial_t w^{\epsilon}( t,
    x)+h_{\epsilon}({t}^{\epsilon},{x}^{\epsilon})w^{\epsilon}( t,
    x)+\mathcal{L}_s(\Omega_1({t}^{\epsilon},{x}^{\epsilon}),
    w^{\epsilon}( t,  x)) \le f^{\epsilon}( t,  x),
\]
where
\begin{align*}
&h_{\epsilon}(t,x)=h_1({t}^{\epsilon},{x}^{\epsilon}),\\
&\Omega_{\epsilon}(t,x)=\Omega_1({t}^{\epsilon},{x}^{\epsilon}),\\
&f^{\epsilon}(\bar t, \bar x)=  f_1(\bar{t}^{\epsilon},\bar{x}^{\epsilon})-f_2(\bar{t}_{\epsilon},\bar{x}_{\epsilon})+M|h_1(\bar{t}^{\epsilon},\bar{x}^{\epsilon})-h_2(\bar{t}_{\epsilon},\bar{x}_{\epsilon})|  +2M\int_{|z|\ge \frac{\zeta}{2}d(\bar x)}\frac{|\chi_{\Omegat_1(\bar{t}^{\epsilon},\bar{x}^{\epsilon})}-\chi_{\Omegat_2(\bar{t}_{\epsilon},\bar{x}_{\epsilon})}|}{|z|^{N+2s}}dz
\end{align*}
and the point ${t}^{\epsilon},{x}^{\epsilon}$ is related to $( t,  x)$ through \eqref{01-06bis} and \eqref{control}. Thanks to Lemma \ref{stability}, we can pass to the limit in \eqref{01-06tris} as $\epsilon\to 0$ and obtain the desired result. 
\end{proof}

For later purpose we also explicitly state an elliptic version of Lemma \ref{differenza}. 
\begin{cor}\label{ellipticdif}
Assume \eqref{alpha}-\eqref{ostationary}, that $f_1,f_2\in C(\Omega)$ satisfy \eqref{fcon} and that $u\in USC_b(\Omega)$, $v\in LSC_b(\Omega)$ solve
\[
h(x)u(x)+\mathcal{L}_{s} (\Omega(x),u(x))\le f_1(x) \qquad \mbox{in } \Omega
\]
\[
h_2(x)v(x)+\mathcal{L}_{s} (\Omega(x),v(x))\ge f_2(x) \qquad \mbox{in } \Omega,
\]
respectively. Then $w=u-v$ solves
\[
\partial_tw(x)+h(x)w(x)+\mathcal{L}_{s} (\Omega(x),w(x))\le f_1(x)-f_2(x) \qquad \mbox{in }  \Omega.
\]
\end{cor}
\begin{remark}
   For the sake of brevity, we do not provide a proof of Corollary \ref{ellipticdif}, see Remark \ref{ellipticparabolic}.
\end{remark}

\subsection{Regularity}
 By adapting the regularity theory for fully nonlinear
 integro-differential equations from  \cite[Sec. 14]{caff} we can
 prove the following. 
\begin{teo}[H\"older regularity]\label{regelliptic}
Let us assume \eqref{alpha}-\eqref{Sigmadef}, that $f\in C(\Omega)\cap\elle{\infty}$, and that $u\in C(\Omega)\cap\elle{\infty}$ solves in the viscosity sense
\[
 h(x)u(x)+\mathcal{L}_{s}(\Omega(x), u(x) )= f(x) \quad \mbox{in } \Omega.
\]
Then, for any open sets $\Omega'\subset\subset \Omega''\subset\subset\Omega$, it follows that
\[
\|u\|_{C^{\gamma}(\Omega')}\le \tilde C,
\]
where $\gamma \in (0,1)$ and $\tilde C=\tilde C(\|f\|_{\elle{\infty}}, s, \zeta, d(\Omega'', \Omega'),\|u\|_{\elle{\infty}})$.
\end{teo}
\begin{proof}
We claim that
\[
 \mathcal{L}_{s}(\Sigma({x}),u(x))= \tilde f(x) \quad \mbox{in } \Omega
\]
 in the viscosity sense, where $\Sigma({x})=\Sigma+x$ and $\tilde f\in
 C(\Omega)$ is a suitable function such that $\tilde f(x)\approx
 d(x)^{-2s}$ close to $\partial\Omega$. Once such a property is
 verified, the proof of the Lemma follows from \cite[Theorem
 4.6]{mou}. See also \cite[Theorem 7.2]{schwabsil}, were the parabolic
 problem is treated.

In order to prove the claim, we follow the ideas of \cite[Sec. 14]{caff}. Let us assume that, any time $u$ is touched from above with a smooth function at some $x\in \Omega$, $u$ belongs to $C^{1,1}( x)$. Using Lemma \ref{continuity}, we deduce that
\[
h({x})u( x)+\mathcal{L}_{s} (\Omega({x}),u( x))\le f(x)
\]
pointwise for any such a $x\in \Omega$. Thanks to assumption \eqref{ostationary}, the nonlocal
operator can be estimated as follows 
\begin{align*}
  &
\mathcal{L}_{s}(\Omega({x}),
     u(x))=\int_{\Sigma\cap B_{\zeta d({x}^{\epsilon})}}\frac{ u(    {x} )- u(x
    +z )}{|z|^{N+2s}}dz+\int_{\Omegat(x)\setminus B_{\zeta d(x)}}\frac{ u(
    x )- u( x+z )}{|z|^{N+2s}}dz\\
  &\quad
=\int_{\Sigma}\frac{ u(    x )- u( x+z
    )}{|z|^{N+2s}}dz+\int_{\Omegat(x)\setminus B_{\zeta d(x)}}\frac{ u(
    x )- u( x+z )}{|z|^{N+2s}}dz-\int_{\Sigma\setminus B_{\zeta d(x)}}\frac{ u(
    x )- u( x+z )}{|z|^{N+2s}}dz\\
  &\quad
=  \mathcal{L}_{s}(\Sigma({x}),
     u(x)).
\end{align*}
Let us set
\[
\tilde f(x)=f(x)-\int_{\Omegat(x)\setminus B_{\zeta d(x)}}\frac{ u(
    x )- u( x+z )}{|z|^{N+2s}}dz+\int_{\Sigma\setminus B_{\zeta d(x)}}\frac{ u(
    x )- u( x+z )}{|z|^{N+2s}}dz.
\]
This proves that
\[
  \mathcal{L}_{s}(\Sigma({x}),
     u(x)) \le \tilde f(x),
\]
assuming that $u$ belongs to $C^{1,1}(x)$. Let us apply this argument to the sup convolution $u^{\epsilon}$.
By definition of the sup convolution, we recall that, any time $u^{\epsilon}$ is touched from above by a smooth function $\varphi$ at $x\in \Omega$, then $u^{\epsilon}\in C^{1,1}(x)$. Moreover, thanks to Theorem \ref{pointsub}, we have that
\[
h(x^{\epsilon})u(x)+\mathcal{L}_{s}(\Omega(x^{\epsilon}), u(x) )\le f(x^{\epsilon})
\]
point wise for any such a $x\in\Omega $. Then, thanks to the argument above, it follows that
\[
  \mathcal{L}_{s}(\Sigma({x^{\epsilon}}),
     u^{\epsilon}(x)) \le \tilde f(x^{\epsilon}).
\]
Eventually, thanks to the stability property of viscosity solution (see Lemma \ref{stability}), we can pass to the limit in the inequality above to conclude that
\[
  \mathcal{L}_{s}(\Sigma({x}),
     u(x)) \le \tilde f(x),
\]
in the viscosity sense. Similarly, one can check that 
\[
 \mathcal{L}_{s}(\Sigma({x}),
     u(x))\ge \tilde f(x),
\]
and the proof of the initial claim follows.
\end{proof}

\subsection{Equivalence with the fractional laplacian}
We now present two technical lemmas, shedding light on  the
relation between the operator in \eqref{10:17} and the classical
 fractional laplacian. 

\begin{lemma}[Equivalence]\label{acca}
The function defined in  \eqref{15:15} can be equivalently written as
\[
a(x)=\Gamma(2s+1)\int_{\tilde{S}(x)^c}\frac{1}{|z|^{N+2s}},
\]
where $S(x)$ is the largest star-shaped  subset of $\Omega$
centered at $x$ and $\tilde{S}(x)=x-S(x)$.

If  $\Omega$ is convex, the fractional laplacian
$(-\Delta)_s$ defined in \eqref{decomposition} is 
equivalent to the elliptic operator defined in \eqref{start},
as  $$\Gamma(2s+1) (-\Delta)_s\varphi(x) = a(x) \varphi (x)+ (-\Delta)^\star_s\varphi (x)$$ on
suitably smooth
function $\varphi$.
\end{lemma}
\begin{proof}
We firstly  notice that
\begin{align*}
a(x)&=\int_{\mathbb{R}^N}\frac{1}{|y|^{N+2s}}e^{-\frac{d(x,\sigma(y))}{|y|}}dy\\
&
=\int_{\omega^{N-1}}\int_0^{\infty}\frac{1}{\rho^{1+2s}}e^{-\frac{d(x,\sigma)}{\rho}}d\rho
  d\sigma=\int_{\omega^{N-1}}\frac{1}{d(x,\sigma)^{2s}}d\sigma
  \int_0^{\infty}\frac{1}{r^{1+2s}}e^{-\frac{1}{r}}dr\\
&
=\int_{\omega^{N-1}}\frac{1}{d(x,\sigma)^{2s}}d\sigma\int_0^{\infty}t^{2s-1}e^{-t}dt=\Gamma(2s)\int_{\omega^{N-1}}\frac{d\sigma}{d(x,\sigma)^{2s}}
\end{align*}
where we recall that $d(x,\sigma(y))$ denotes the distance
between $x$ and the first point reached on $\partial \Omega$ by the ray from
$x$ with direction $\sigma(y) = y/|y|$.
On the other hand, we have that
\[
\int_{\tilde{S}(x)^c}\frac{1}{|z|^{N+2s}}dz=\int_{\omega^{N-1}}\int_{d(x,\sigma)}\rho^{-1-2s}d\rho d\sigma=\frac{1}{2s}\int_{\omega^{N-1}}\frac{d\sigma}{d(x,\sigma)^{2s}}.
\] 
The conclusion follows from the fact that $\Gamma (2s+1)=\Gamma(2s)2s$.\\

Assume now that $\Omega$ is convex. Then $S(x)\equiv\Omega$ for any $x\in\Omega$ and
\[
a(x)=\Gamma(2s+1)\int_{\{x-\Omega\}^c}\frac{1}{|z|^{N+2s}}dz=\Gamma(2s+1)\int_{\Omega^c}\frac{1}{|x-y|^{N+2s}}dy.
\]
Then, recalling \eqref{start}, we have that, for any $\varphi\in C^{\infty}_c(\Omega)$,
\[
a(x)\varphi(x)+(-\Delta)_{s}^{\star}\varphi(x)=\Gamma(2s+1)\left[\int_{\Omega^c}\frac{1}{|x-y|^{N+2s}}dy \ \varphi(x)+p.v.\int_{\Omega}\frac{\varphi(x)-\varphi(y)}{|x-y|^{N+2s}}dy\right]
\]
\[
=\Gamma(2s+1)\ p.v.\int_{\mathbb{R}^N}\frac{\varphi(x)-\varphi(y)\chi_{\Omega}}{|x-y|^{N+2s}}dy=\Gamma(2s+1)(-\Delta)_ {s}\varphi (x).\qedhere
\]
\end{proof}

We now introduce a class of domains for which the
function $k(x)$ defined in \eqref{decomposition} satisfies the
bounds \eqref{killing}. To this aim, we assume that the
complement $\Omega^c $  satisfies a uniform positive density
condition, namely that there exists $\rho_0>0$ and
$\kappa >0 $ such that
\be\label{density}
|B_{\rho}(\bar x)\cap \Omega^c|\ge \kappa  |B_{\rho}(\bar x)| \ \ \ \mbox{for all } \bar x\in\partial\Omega \ \mbox{ and } \ \rho\in (0,\rho_0).
\ee
Let us stress that \eqref{density} is weaker than the  exterior cone condition.

\begin{lemma}[Bounds on $h$]\label{hbeha}
Let $\Omega$ be an open bounded set of $\mathbb R^N$ with
$\Omega^c$ satisfying  \eqref{density}. Then, the function 
\[
h (x)=\int_{\Omega^c}\frac{1}{|x-y|^{N+2s}}dy
\]
satisfies the bounds  \eqref{killing}.
\end{lemma}
\begin{proof}
To start with,  notice that, we have $\Omega^c\subset
B_{d(x)}(x)^c$ for all $x\in\Omega$. Then,
\be\label{easypart}
h(x)\le \int_{B_{d(x)}(x)^c}\frac{1}{|x-y|^{N+2s}}dy=\int_{|z|\ge d(x)}\frac{1}{|z|^{N+2s}}dz=\frac{\omega_N}{2s}\frac1{d(x)^{2s}}
\ee
whence the upper bound in  \eqref{killing}. 

To prove the lower bound  let us take
$x\in\Omega$ such that $d(x)\le \rho_0$. We then have  % It results that
\[
h(x)=\int_{\Omega^c}\frac{1}{|x-y|^{N+2s}}d y  \ge
\int_{\Omega^c\cap B_{d(x)}(\bar x)}\frac{1}{|x-y|^{N+2s}}d y 
\]
where $\bar x\in\partial\Omega$ is such that $d(x)=|x-\bar x|$. Using that if $y\in B_{d(x)}(\bar x)$ then $|x-y|\le |x|+|y|\le 2d(x)$ and taking advantage of \eqref{density} (recall that $d(x)\le \rho_0$), the inequality above becomes
\[
h(x)\ge\frac{1}{(2d(x))^{N+2s}}|\Omega^c\cap B_{d(x)}(\bar x)|\ge C\frac{1}{d(x)^{2s}}.\qedhere
\]
\end{proof}

\section{Existence and uniqueness}

The existence of viscosity solutions follows by applying the
classical Perron method. We give here full details of this
construction in the parabolic case, hence proving Theorem
\ref{existence}. The proof of Theorem \ref{teide} follows the same
lines, being actually simpler. We comment on it at the end of the
section.

As a first step toward the implementation of the Perron method,
we start by providing  a suitable barrier for the elliptic problem. %, then we prove a comparison result for sub and super solutions that are ordered at the boundary and, finally, we implement to our setting the Perron's method.

\begin{lemma}[Barriers]\label{barriercone}
Let us assume that the set valued function $x\to\Omega(x)$ satisfies \eqref{ostationary}, \eqref{Sigmadef}. Then there exists a positive $\bar \eta=\bar \eta(N,s,\zeta,\alpha)$ such that, for any $\eta\in(0,\bar \eta]$,
the function $u_{\eta}(x)=d(x)^{\eta}$ solves the inequality
\be\label{sbomba}
\frac{\alpha}{d(x)^{2s}}u_{\eta}(x)+\mathcal{L}_s(\Omega(x),u_{\eta}
(x) )\ge \frac{\alpha}{2} d(x)^{\eta-2s} \ \ \ \mbox{in} \ \ \ \Omega
\ee
in the viscosity sense.
\end{lemma}
\begin{remark}\label{barpar}\rm
Notice that under  assumptions \eqref{Sigmadef} and \eqref{omegax}, the function $u_{\eta}$, with $\eta<\bar \eta(N,s,\zeta_T,\alpha)$ also satisfies in the viscosity sense
\[
\partial_tu_{\eta}(x)+ \frac{\alpha}{d(x)^{2s}}u_{\eta}(x)+\mathcal{L}_s(\Omega(t,x),u_{\eta}
(x) )\ge \frac{\alpha}{2} d(x)^{\eta-2s}  \ \ \ \mbox{in} \ \ \ (0,T)\times\Omega,
\]
since for all $t\in(0,\infty)$ the set valued function $x\to\Omega_t(x)=\Omega(t,x)$ satisfies \eqref{ostationary} with $\zeta_T$. If we moreover assume \eqref{threeprime}, the barrier is uniform in time\EEE.
\end{remark}
\begin{proof}[Proof of Lemma \ref{barriercone}]
Fix  $x\in\Omega$ and assume that there exists $\varphi\in
C^2(\Omega)$  such that $u_{\eta}-\varphi$ has a minimum in $\Omega$
at $x$ and that $u_{\eta}(x)=\varphi(x)$. We have to check (see
Lemma  \ref{def2}) that for any $B_r(x)\subset \Omega$
\be\label{11:32}
\frac{\alpha}{d(x)^{2s}}u_{\eta}(x)+\mathcal{L}_s(\Omega(x),\varphi_r(x))\ge \frac{\alpha}{2} d(x)^{\eta-2s}.
\ee
Since  $\varphi$ touches $u_{\eta}$ from below  at $x$, we
deduce \cite[Prop. 2.14]{BC} that there exists a unique
$\bar x\in\partial \Omega$ such that $d(x)=|x-\bar x|$. To simplify
notation, from now on we use a system of coordinates that is
centered at $\bar x$, so that $d(x)=|x|$. Notice that %, thanks to the properties of $\varphi$, it results that
\[
\varphi(x)-\varphi(x+z)\ge u_{\eta}(x)-u_{\eta}(x+z)\ge (|x|^{\eta}-|x+z|^{\eta}) \ \ \ \mbox{for all} \ \ \ z\in\Omegat(x),
\]
the last inequality following from the fact that $d(x+z)\le
|x+z|$. We then have that  %Then it results that
\[
\mathcal{L}_s(\Omega(x),\varphi_r(x))\ge  \int_{\Omegat(x)}\frac{|x|^{\eta}-|x+z|^{\eta}}{|z|^{N+2s}}dz
\]
\[
=\left(\int_{ \{|z|\le \zeta
    |x|\}\cap\Sigma }\frac{|x|^{\eta}-|x+z|^{\eta}}{|z|^{N+2s}}dz+\int_{\{|z|>
    \zeta
    |x|\}\cap\Omegat(x)}\frac{|x|^{\eta}-|x+z|^{\eta}}{|z|^{N+2s}}dz\right)=:
I_1+I_2,
\]
 where we have used  that $\Omegat(x)\cap B_{\zeta
    |x|}=\Sigma\cap B_{\zeta|x|}$ (see assumption \eqref{ostationary})\EEE.  Thanks to Taylor expansion, we get that
\begin{align*}
|x+z|^{\eta}&=|x|^{\eta}+\eta |x|^{\eta-2}x\cdot z
+\frac12\left[\eta(\eta-2)|\xi|^{\eta-4}|\xi \cdot z|^2+\eta
  |\xi|^{\eta-2}|z|^2 \right]\\
&
\le |x|^{\eta}+\eta |x|^{\eta-2}x\cdot z+2^{\eta-3}\eta |x|^{\eta-2}|z|^2  \ \ \ \mbox{ for } \ \ |z|\le \zeta |x|
\end{align*}
with $\xi= x+t z$ for some $t\in(0,1)$ and the 
inequality follows   from neglecting a negative term and
from  the fact that $|\xi|\ge (1-\zeta)|x|\ge \frac12 |x|$. This implies that
\begin{align}\label{16:42bis}
I_1&\ge-\int_{\{|z|\le \zeta |x|\}\cap\Sigma}\left(\eta |x|^{\eta-2}x\cdot z+\eta
  2^{\eta-3} 
  |x|^{\eta-2}|z|^2\right)\frac{1}{|z|^{N+2s}}dz\\
&
=- \eta 2^{\eta-3}  |x|^{\eta-2}\int_{\{|z|\le \zeta |x|\}\cap\Sigma}|z|^{2-N-2s}dz\ge- \eta C|x|^{\eta-2s}. \nonumber
\end{align}
 where, in the second line, we have used that the set $\{|z|\le \zeta |x|\}\cap\Sigma$ is radially symmetric and that the first order term of the expansion vanishes in the principal value sense\EEE.
On the other hand, one has that 
\be\label{16:43bis}
I_2\ge|x|^{\eta-2s}\int_{\{|y|\ge \zeta\}}\frac{1-(1+|y|)^{\eta}}{|y|^{N+2s}}dy.
\ee
%Let us focus on the integral in the right-hand side above:
%\begin{align*}
%&\int_{\{|y|> \zeta\}}\frac{(1+|y|)^{\eta}-1}{|y|^{N+ 2
%  s}}dy=\omega_N\int_{\zeta}^{\infty}\frac{(1+t)^{\eta}-1}{t^{ 2
%  s-1}}dt\\
%&\quad
%=
%  \omega_N\sum_{i=0}^{\infty}\int_{{\zeta}^{1-i}}^{{\zeta}^{-i}}\frac{(1+t)^{\eta}-1}{t^{2s+1}}dt\le \omega_N\sum_{i=0}^{\infty}\frac{(1+{\zeta}^{-i})^{\eta}-1}{{\zeta}^{(1-i)(2s+1)}}({\zeta}^{-i}-{\zeta}^{1-i})\\
%&\quad
%\le C\sum_{i=0}^{\infty}\left( \zeta^{i2(s-\eta)}-\zeta^{i2s}\right)\le C\left(\frac{1}{1-\zeta^{2(s-\eta)}}-\frac{1}{1-\zeta^{2s}}\right).
%\end{align*}
Combining this last inequality with \eqref{16:42bis} and \eqref{16:43bis}, we obtain that
\[
\mathcal{L}_s(\Omega(x),\varphi_r(x))\ge- g(\eta)d(x)^{\eta-2s},
\]
where $$g(\eta)=C\left(\eta+\int_{\{|y|\ge \zeta\}}\frac{1-(1+|y|)^{\eta}}{|y|^{N+2s}}dy\right).$$
Notice that, thanks to the  Lebesgue Dominated Convergence
Theorem,  the integral in the brackets above  goes to zero as $\eta\to0$. It follows  that
\[
\frac{\alpha}{
  d(x)^{2s}}u_{\eta}(x)+\mathcal{L}_s(\Omega(x),\phi_r(x))-f(x)\ge   (\alpha-g(\eta))d(x)^{\eta-2s}.
\]
At this point it is enough to chose $\eta\le\bar \eta$  satisfying
$
\alpha-g(\bar \eta)= {\alpha}/2$ in order to conclude the
proof. 
\end{proof}

 Let us now provide a comparison principle for equation
\eqref{parabolic}.  This relies on  Lemma \ref{differenza},
 which is in turn based  on the regularization of sub/super-solutions through sup/inf convolution.

\begin{lemma}[Comparison] \label{timecomp}
Assume \eqref{Sigmadef}, \eqref{contpar}-\eqref{alphap}, that $T\in (0,\infty)$, that $u(t,x)$ and $v(t,x)$ are sub- and
supersolutions to \eqref{parabolic}, respectively,  that they
are ordered on the boundary, namely $u\le v$ on $
(0,T)\times\partial\Omega  $, and that $u(0,\cdot)\le v (0,\cdot)$ on
$\Omega$.  Then,
\begin{equation}
u\le v \ \ \mbox{in} \ \  (0,T)\times\Omega  .
\end{equation}
\end{lemma}
\begin{proof}
Given $\delta>0$, let  us introduce the function $u_{\delta}(t,x)=u(t,x)-\frac{\delta}{T-t}$ and notice that it is a  viscosity subsolution to \eqref{parabolic}, namely,
\[
\partial_t u_{\delta}(t,x) +h(t,x)u_{\delta} (t,x) +\mathcal{L}_{s}(\Omega(t,x),
u_{\delta} (t,x)  )- f(t,x)\le -\frac{\delta}{T^2}<0.
\]
We firstly show  that $u_{\delta}\le v$, for any $\delta>0$,
and then conclude the proof by  taking the limit as $\delta$
goes to $0$. Using Lemma \ref{differenza}, we deduce that $w=u_{\delta}-v$ solves in the viscosity sense
\[
\partial_t w(t,x) +h(t,x)w(t,x)+\mathcal{L}_{s}(\Omega(t,x),
w (t,x)  )\le0 \ \ \  \mbox{in} \ \Omega.
\]
Let us assume by contradiction that $\sup_{(0,T)\times\Omega}w=
M >0$. Due to the ordering assumption on the parabolic boundary $
(0,T)\times\partial\Omega  $ and on the initial conditions, and the
behavior of $u_{\delta}$ as $t\to T^-$,  $M$ is
attained inside at $(\bar t, \bar x)\in(0,T)\times\Omega$. This 
implies  that the constant function $M$ touches from above $w$ at
 the  point $(\bar t, \bar x)$, and then it is an admissible test function for $w$ to be a subsolution. It follows that
\[
\frac{\alpha}{d^{2s}(\bar x)}M\le h(\bar t, \bar x)M\le 0,
\]
that is clearly a contradiction. Then $u_{\delta}-v=w\le 0$ for any
$\delta>0$ and  the assertion follows. \EEE
\end{proof}

\begin{cor}[Elliptic comparison] \label{ellcomp}
Assume \eqref{alpha}-\eqref{Sigmadef} that $u(x)$ and $v(x)$ are sub- and
supersolutions to \eqref{10:17}, respectively, and that  $u\le v$ on $\partial\Omega$. Then,
\begin{equation}
u\le v\ \ \mbox{in} \ \  \Omega  .
\end{equation}
\end{cor}
 The proof of this corollary can be easily deduced from that of
Corollary \ref{ellipticdif} and we omit the details for the sake of brevity.

We are now ready to present a first existence result, which
relies on the possibility of finding suitable barriers for  the parabolic problem. We will later
check that such barriers can be easily obtained from Lemma \ref{barriercone}.

\begin{teo}[Existence, given barriers] \label{perronpar}Assume \eqref{Sigmadef}, \eqref{contpar}-\eqref{alphap}.
Let  $T \in (0,\infty) $  and    $\underline
l(t,x)$ and $\overline l(t,x)$ be  sub- and supersolution to
\eqref{parabolic}, respectively, with  $\underline
l=\overline l=0$ on $(0,T) \times  \partial
\Omega$. Then, for any $u_0\in C(\Omega)$ such that $\underline
l(0,x)\le u_0(x) \le \overline l(0,x)$ for all $x\in\Omega$, problem
\eqref{parabolic}  admits a unique viscosity solution.
\end{teo}

\begin{proof}
We aim at applying  Perron's method.  Let us set
\begin{align*}
 A&=\Big\{w\in \USC(\Omega\times(0,T)) \ : \  \underline l(t,
    x)\le w(t,x)\le \overline l(t,  x) \ \mbox{for $(t,
     x) \in (0,T) \times \partial \Omega$, } \\
&\hspace{25mm}  w \ \mbox{is a  subsolution to \eqref{parabolic}, and } w(x,0)\le u_0(x) \Big\}.
\end{align*}
Since $\overline l \in A\not = \emptyset$,  we can set
\[
u(t,x)=\sup_{w\in A} w(t,x).
\] 
By definition, it follows \EEE that for any $(\bar t, \bar
x)\in(0,T)\times\Omega$ there exists a sequence $\{v_n\}\subset A$
that $\Gamma-$converges to the uppersemicontinuous envelop $u^*$ at $(\bar t, \bar x)$. We \EEE can
use Lemma \ref{stability} to show that $u^*$ is a subsolution to the equation in 
\eqref{parabolic}.  Moreover $u^*(\cdot,0)\le u_0(\cdot)$ in $\Omega$
and $u^*(t,\cdot)\le 0$ on $\partial\Omega$ for any $t\in(0,T)$. In
fact, \EEE assume by contradiction that there exists $\bar x \in \Omega$ such that $u^*(x,0)- u_0(x)=\xi>0$. This would mean that there exist $\{x_n\}\subset\Omega$, $\{t_n\}\subset [0,T)$ and $\{w_n\}\subset A$ such that
\[
x_n\to \bar x, \ \ \ t_n\to 0 \ \ \ \mbox{and} \ \ \ w_n(x_n,t_n)\to u_0(\bar x)+\xi,
\]
that is in contradiction with the definition of $u$. Similarly we
check that $u^*(t,\cdot)\le 0$ on $\partial\Omega$. This implies that $u^*\in A$ and, by definition of $u$, we get that $u=u^*$

Now we claim that the lower-semicontinuous envelope \EEE $u_*$ is a supersolution to \eqref{parabolic} and that $u_*(x,0)\ge u_0(x)$ for all $x\in\Omega$. Once the claim is proved, we can apply
the comparison principle of Lemma \ref{timecomp} to the subsolution
$u$ and the \EEE  supersolution $u_*$  to infer that 
\[
 u\le u_* .
\]
This implies that $u_*=u^*=u$ is a viscosity solution to
\eqref{parabolic} that satisfies the boundary and initial
conditions \EEE in the classical sense.\\

Let us hence prove that $u_*$ is a supersolution with
$u_*(x,0)\geq u_0(x)$. \EEE
%Let us show at first that $u_*$ is a supersolution to
%\eqref{parabolic}. 
By contradiction, we assume there exists $\phi\in  C^2( (0,T)\times\Omega  )$ such that $u_*-\phi$ has a strict global minimum at $(t_0,x_0)$, $u_*(t_0,x_0)=\phi(t_0,x_0)$ and 
\be\label{16:03}
\partial_t \phi(t_0,x_0)+ h(t_0,x_0)u_*(t_0,x_0)+\int_{\Omega(t_0,x_0)}  \frac{\phi(t_0,x_0)-\phi(z,t_0)}{|x_0-z|^{N+2s}}dz<f(t_0,x_0).
\ee
This means that there exists $\epsilon>0$ such that the function 
\[
F(t,x)=\partial_t \phi(t,x)
h(x)u_*(t,x)+h(t,x)\phi(t,x)+\mathcal{L}_{s} (\Omega(t,x),\phi
(t,x) \EEE)-f(t,x)
\] satisfies $F(t_0,x_0)=-\epsilon$. Since such a function is
continuous at $(t_0,x_0)$, \EEE there exists $r>0$ such that
$F(t,x)<-\frac{\epsilon}{2}$ for all $(t,x)\in
\overline{B_r(t_0,x_0)}$ where $B_{r}(t_0,x_0)=\{(|t- t_0|^2+ |x-
x_0|^2)^{\frac12}<r\}\subset (0,T)\times\Omega$. Let us define \EEE
\[
\delta_1=\inf_{x\in\Omega \setminus B_{r}(t_0,x_0)} (v- \phi)(t,x)>0, \ \ \  \ \ \ \delta_2=\frac{\epsilon}{4 \sup_{x\in {B_r(x_0)}}h(x)},
\]
and set $\delta=\min\{\delta_1,\delta_2\}$. With this choice of $r$ and $\delta$ we define
\[
V =
\begin{cases}
\max\{v ,{\phi}+\delta\} \qquad & \mbox{in } B_{r}(t_0,x_0),\\
 v  & \mbox{otherwise}\, .
\end{cases}
\]
Notice \EEE that, since $v$ is upper semincontinuous, \EEE
the set $\{v -{\phi}-\delta<0\}$ is open and nonempty. %: open since
                                %$v$ is upper semicontinuous, non
                                %empty 
(since, by definition of lower semicontinuous envelop, there exists a sequence $x_n\to x_0$ such that $v(z_n)\to u_*(x_0)=\phi(x_0)$ as $n\to\infty$). Moreover $\{v -{\phi}-\delta<0\}\subset B_{r}(t_0,x_0) $ thanks to the choice of $\delta_1$.
We want to prove that $V $ is a subsolution. Let us consider now
$\psi\in  C^2( (0,T)\times\Omega  )$ such that $V -\psi$ has a global
maximum at $(\tau_0,y_0)$ \EEE and $\psi(\tau_0,y_0)=V
(\tau_0,y_0)$. If $V (\tau_0,y_0)=v(\tau_0,y_0)$, since $V \ge v$, it
results that $v -\psi$ has \EEE a global maximum at $(\tau_0,y_0)$ and $v(\tau_0,y_0)=\psi(\tau_0,y_0)$. Using that $v $ is a subsolution we get that
\begin{align*}
&\partial_t \psi(\tau_0,y_0)+
  h(y_0)V(\tau_0,y_0)+\mathcal{L}_s(\Omega(\tau_0,y_0),\psi
  (\tau_0,y_0) \EEE)\\
&\quad
=\partial_t \psi(\tau_0,y_0)+h(y_0)
  v(\tau_0,y_0)+\mathcal{L}_s(\Omega(\tau_0,y_0),\psi(\tau_0,y_0)
  \EEE) \le f(\tau_0,y_0).
\end{align*}

Let us now focus on the case $V (\tau_0,y_0)=\phi(\tau_0,y_0)+\delta\neq v(\tau_0,y_0)$. This implies that 
\[
\partial_t \psi(\tau_0,y_0)=\partial_t \phi(\tau_0,y_0).
\]
 and that $(\tau_0,y_0)\in B_{r}(t_0,x_0)$. Then we have that
\[
\phi+\delta-\psi\le V -\psi\le 0 \ \ \ \mbox{in} \ \ B_{r}(t_0,x_0),
\]
where we have used the fact that $\phi+\delta\le V $ in
$B_{r}(t_0,x_0)$. Moreover, we readily check \EEE 
\[
\phi+\delta-\psi\le \phi+\delta-v \le 0 \ \ \ \mbox{in} \ \  (0,T)\times\Omega  \setminus B_{r}(t_0,x_0),
\]
since $v \equiv V  \le \psi$ in $ (0,T)\times\Omega  \setminus
B_{r}(t_0,x_0)$ and thanks to the definition of $\delta_1$. As
effect of \EEE the two inequalities above, we deduce that
$\phi+\delta\le \psi$ in $ (0,T)\times\Omega  $. It follows that \EEE
\begin{align*}
 &\partial_t \psi(\tau_0,y_0)+
   h(\tau_0,y_0)V(\tau_0,y_0)+\mathcal{L}_s(\Omega(\tau_0,y_0),\psi
   (\tau_0,y_0) \EEE)
  \\
&\quad
\le\partial_t \phi(\tau_0,y_0)
  +h(\tau_0,y_0)(\phi(\tau_0, \EEE y_0)+\delta)+\mathcal{L}_s(\Omega(\tau_0,y_0),\phi (\tau_0,y_0) \EEE)
  \\
&\quad \le f(\tau_0,y_0)-\frac{\epsilon}2+\frac{\epsilon}4<f(\tau_0,y_0),
\end{align*}
where the last inequality comes from the choice of $r$ and
$\delta$. This leads \EEE to a contradiction since it
implies \EEE that $V\in A$ and that $V>v\ge u$ somewhere in $
B_{r}(t_0,x_0)$. This proves that \EEE  $u_*$ is a supersolution to \eqref{parabolic}.

Finally let us \EEE  prove that $u_*(x,0)\ge u_0(x)$ for all $x\in\Omega$. Again assume by contradiction that there exists $\bar x\in\Omega$ such that
\be\label{spring}
u_*(\bar x,0)< u_0(\bar x).
\ee
Our aim is to build \EEE a barrier from \EEE below for
$u(t,x)$ in a neighborhood of $(0,\bar x)$ (hence a barrier for
$u_*$, as well), contradicting \EEE \eqref{spring}.
Thanks to the continuity of $u_0$, for any $\epsilon>0$ there exists $\delta_{\epsilon}< \frac12 d(\bar x)$ such that
\[
|u_0(\bar x)-u_0(x)|\le\epsilon \ \ \ \mbox{if} \ \ \ |\bar x -x|\le \delta_{\epsilon}.
\]
Take now a function $\eta(x)\in C_c^{\infty}(B_1(0))$ with $0\le \eta\le 1$ and $\eta(0)=1$, and define
\[
\tilde w(t,x) =a\eta\left(\frac{\bar x -x}{\delta_{\epsilon}}\right)-b-K\delta_{\epsilon}^{-2s}t,
\]
with $a=u_0(\bar x)-\epsilon+\|u_0\|_{\elle{\infty}}$,
$b=\|u_0\|_{\elle{\infty}}$, and $K>0$ to be chosen below. \EEE
Thanks to the choice of $a$ and \EEE $ b$ it is easy to check that $\tilde w(0,x)\le u_0(x)$. Moreover, recalling that supp$\left(\eta\left(\frac{\bar x -x}{\delta_{\epsilon}}\right)\right)\subset B_{\delta_{\epsilon}}(\bar x)$ and that the integral operator scales as $\delta^{-2s}$, we get that
\[
\partial_t \tilde w+h \tilde w+\mathcal{L}_s(\Omega(t,x),\tilde w
(t,x) \EEE )-f\le -K\delta_{\epsilon}^{-2s}+\delta^{-2s}C(\eta)+\|f\|_{\elle{\infty}}\le 0,
\]
 where the last inequality follows by letting \EEE  $K>
 C(\eta)+\delta_{\epsilon}^{2s}[aC(d(\bar
 x))+\|f\|_{\elle{\infty}}]$. Hence, \EEE $\tilde w\in A$ and, by
 definition of $u(t,x)$, $\tilde w(t,x)\le u(t,x)$. %Since $\tilde
                                %w(t,x)$ is continuous and $\tilde
                                %w(0,x)=u_0(\bar x)-\epsilon$, it is
                                %the suitable barrier that we were
                                %searching for. Indeed
Now, \EEE for any $\epsilon>0$, there exists $\tilde{\delta}_{\epsilon}$ (possibly smaller then $\delta_{\epsilon}$) such that
\[
u_0(\bar x)-2\epsilon\le \tilde w(t,x) \le u(t,x) \ \ \ \mbox{for} \ \ \ |\bar x -x|\le \tilde{\delta}_{\epsilon} \quad t\in [0,\tilde{\delta}_{\epsilon}).
\]
Then the same inequality holds for $u_*$, contradicting \EEE \eqref{spring}.
\end{proof}

\begin{proof}[Proof of Theorem \ref{existence}] Let us argue
  for $T < \infty$ first. Choose  $\eta\le
  \min\{\bar{\eta},\eta_1\}$  where $\eta_1$ is from
  \eqref{fconintro}-\eqref{u0con} and  $\overline \eta $ from  Lemma
  \ref{barriercone} and Remark \ref{barpar},
  and set $\overline l(t,x)= Q d(x)^{\eta}$ where $Q$ is a positive
  constant to be chosen later.   Whenever a smooth function
  $\varphi$ touches $\overline l$ from above  at $(t_0,x_0)$ we deduce that
\begin{align*}
&\partial_t\varphi_r(t_0,x_0)+h(t_0.x_0)\varphi_r(t_0,x_0)+\mathcal{L}_{s}
  (\Omega(t,x),\varphi_r(t_0.x_0))- f(t_0,x_0)\\
&\quad 
\ge\frac{\alpha}{d(x)^{2s}}\varphi_r(t_0,x_0)+\mathcal{L}_{s}
  (\Omega(t,x),\varphi_r(t_0,x_0))- |f(t_0,x_0)|\\
&\quad
\ge \left(Q \frac{\alpha}{2}-|f(t_0,x_0)|d(x_0)^{2s-\eta}\right)d(x_0)^{\eta-2s}\ge 0.
\end{align*}
The first inequality  comes from the fact that
$\partial_t\varphi_r(t_0,x_0)$ must be zero %(see Remark
% \ref{ellipticdef})
and from  assumption \eqref{alphap}  whereas the  second inequality 
follows by construction of $\overline l$ and  by  Lemma
\ref{barriercone}. The third \EEE inequality follows from
 the  assumption on $f$ (see \eqref{fconintro}) and by 
taking $Q$ large enough. This proves that $\overline{l}(t,x)$ is a
supersolution of \eqref{parabolic}. Similarly, we can show that
$\underline{l}(t,x)=-\overline{l}(t,x)$ is a subsolution. By
possibly taking an even   larger value of $Q$ if necessary, we
deduce that $\underline l(0,x)\le u_0(x) \le \overline l(0,x)$, thanks
to assumption \eqref{fconintro} on $u_0$ and to  the choice
of $\eta$. At this point, we can apply Theorem \ref{perronpar} and
conclude the proof.

The limiting case $T=\infty$ can be tackled by passing to the
limit in the   the sequence $\{u_n\}$ of solutions of problem
\eqref{parabolic} in $(0,n)\times\Omega$. Thanks  to Lemma
\ref{timecomp} we have that  %it follows that
\[
u_n(t,x)\equiv u_m(t,x) \ \ \ \mbox{in} \ \ (0,\min\{n,m\})\times\Omega.
\]
Then, for any $(t,x)\in(0,\infty)\times\Omega$, we can uniquely define
$u(t,x)=u_{[t]+1}(t,x)$, where $[t]$ is the integer part of $t$.
From the comparison principle applied on each domain
$(0,n)\times\Omega$, this uniquely defines a solution for all
times.  %. $(0,n)\times\Omega$.Clearly $u(t,x)$ is the searched solution. It is unique thanks to the comparison principle applied on each $(0,n)\times\Omega$.
\end{proof}

As mentioned above, we are not giving the details of the proof of
Theorem \ref{teide}. Indeed, the elliptic case of Theorem \ref{teide}
 follows again from  by  Perron method, by means of the barriers 
from  Lemma \ref{barriercone}. Here, one is asked to use an elliptic version of the
comparison Lemma \ref{timecomp}, which can be deduced  using Corollary \ref{ellipticdif}.

\section{The eigenvalue problem}\label{seceigenvalue}
In this section,  we focus on \EEE the eigenvalue problem
associated to the operator \EEE \eqref{10-6}. Before discussing
our specific \EEE notion of eigenvalue, we prepare some technical tools.

\begin{lemma}[Strong maximum principle]\label{strongmax}
Assume \EEE\eqref{alpha}-\eqref{Sigmadef} and let \EEE
 $u\in \LSC_b(\Omega)$ \EEE solve $$h(x) u(x)+\mathcal{L}(\Omega(x),u(x))\ge 0$$ in the viscosity sense in $\Omega$ and $u\ge0$ in $\partial\Omega$. Then, either $u\equiv0$ or $u>0$ in $\Omega$.
\end{lemma}
\begin{proof} Notice that, thanks to the comparison principle, we
  have that $u\ge0$ in $\Omega$. Let us assume that $u(x_0)=0$ at some $x_0\in\Omega$ and that, by contradiction, $u(y_0)>0$, for some $y_0\in\Omega$. 

If $y_0\in \Omega(x_0)$, since $x_0$ is a minimum for $u$, there
exists $\varphi\in C^2(\Omega)$ such that $\varphi(x_0)=u(x_0)=0$,
$\varphi(x_0)\le u(x_0)$ in $\Omega$. Moreover, since  $u\in \LSC_b(\Omega)$\EEE,  we can chose
$\varphi$ nonnegative and nontrivial in $\Omega(x_0)$. Since $\varphi$
is an admissible test function for $u$ at point $x_0$ and it follows that
\[
\int_{\Omega(x_0)}\frac{-\varphi(x_0+z)}{|z|^{N+2s}}\ge0.
\] 
This is however contradicting the fact that $\varphi\ge0$ is
nontrivial in $\Omega(x_0)$ and proves that $u(x_0)=0$
implies $u=0$ in $\Omega(x_0)$.

If $y_0\notin \Omega(x_0)$, thanks to assumption
\eqref{Sigmadef} and the fact that $\Sigma$ is open, there exists a finite set of points
$\{x_i\}_{i=0}^K\subset\Omega$ such that $x_i\in \Omega(x_{i-1})$ for
$i=1,\cdots, K$ and $y_0\in\Omega(x_K)$. Using inductively the
previous part we deduce that $u=0$ in each $\Omega(x_i)$, that is
$u(y_0)=0$, which is again a contradiction.
\end{proof}
 The next technical Lemma allows us to restrict the operator 
to a  subdomain of $\Omega$. This requires to modify both the sets
$\Omega(x)$ and the function $h$.  Thanks  to the assumptions,
in particular the density bound for $\Sigma$ in \eqref{Sigmadef}, it
turns out the the restricted operator satisfies the same 
properties  of the original one\EEE.
\begin{lemma}[Localization]\label{localization}
Let $f\in C(\Omega)$ and assume that $v$ solves in a viscosity sense
\[
h(x)v(x)\EEE+\mathcal{L}_s(\Omega(x),v(x)\EEE)\le f(x)\EEE \ \ \ \mbox{in} \ \ \Omega.
\] 
If the the open set $O\subset\Omega$ is such that $v\le 0$ in $\Omega\setminus O$, then $v$ also solves in the viscosity sense
\[
j(x)v(x)\EEE+\mathcal{L}_s(\Xi(x),v(x)\EEE)\le f(x)\EEE \ \ \ \mbox{in} \ \ O,
\]
where $\Xi(x)=\Omega(x)\cap O$ and $j(x)=h(x)+\int_{\Omega(x)\setminus
  O}\frac{dy}{|x-y|^{N+2s}}$.  By additionally assuming  % Let
                                % us assume moreover
                                % \eqref{ostationary} and
                                % \eqref{Sigmadef} and
that $O$ coincides with some ball $\tilde B\subset \Omega$  and by
setting  $\tilde d(x)=\mbox{dist}(x,\partial \tilde B)$, it holds true that
\be\label{restrictedh}
c_1\tilde d(x)^{-2s}\le j(x)\le c_2\tilde d(x)^{-2s} \ \ \ x\in \tilde B.
\ee
\end{lemma}
\begin{proof}
Let us assume that $\max_{O}( v-\varphi)=(v-\varphi)(\bar x)=0$ and
that $B_r(\bar x)\subset\subset O$. It \EEE is possible to extend
$\varphi$ to all $\Omega$ so that $\max_{\Omega} (
v-\varphi)=(v-\varphi)(\bar x)=0$ (with a slight abuse of notation,
we still indicate \EEE the extension by \EEE $\varphi$). Then, we have
\begin{align*}
&f(\bar x)\ge h(\bar x)v(\bar x)+\int_{B_r(\bar x)\cap \Omega(\bar
  x)}\frac{\varphi(\bar x)-\varphi(y)}{|\bar
  x-y|^{N+2s}}dy+\int_{\Omega(\bar x)\setminus B_r(\bar
  x)}\frac{v(\bar x)-v(y)}{|\bar x-y|^{N+2s}}dy\\
&\quad
=\left[h(\bar x)+\int_{\Omega(x)\setminus
  O}\frac{dy}{|x-y|^{N+2s}}\right]v(\bar x)+\int_{B_r(\bar x)\cap
  \Sigma(\bar x)}\frac{\varphi(\bar x)-\varphi(y)}{|\bar
  x-y|^{N+2s}}dy+\int_{\Xi(\bar x)\setminus B_r(\bar
  x)}\frac{v(\bar x)-v(y)}{|\bar x-y|^{N+2s}}dy\\
&\quad 
-\int_{\Omega(x)\setminus O}\frac{v(y)}{|x-y|^{N+2s}}dy\ge \left[h(\bar x)+\int_{\Omega(x)\setminus O}\frac{dy}{|x-y|^{N+2s}}\right]v(\bar x) +\int_{\Xi(\bar x)}\frac{\varphi_r(\bar x)-\varphi_r(y)}{|\bar x-y|^{N+2s}}dy,
\end{align*}
where the last \EEE inequality comes from the fact that \EEE $v\le 0$ in $\Omega\setminus O$.

Let us consider now the case $O\equiv \tilde B$. The estimate from
above in \eqref{restrictedh} can be deduced as in
\eqref{easypart}.  We omit the \EEE details. To show the estimate from below, fix $k>1$ so that
\[
c- \frac 2{\omega_n}|B_1(0)\cap A(k)|\ge \frac12 c,
\]
where $c$ is the constant in \eqref{Sigmadef} and $A(k)=\{y\in\mathbb{R}^N \ : \ -k^{-1}\le y_1\le 0\}$. We also point out that the symmetry of $\Sigma$ implies, for $B^{\pm}_r(0)=B_r(0)\cap \{z\le0\}$ and $r>0$, that
\be\label{auxilia}
|\Sigma\cap B^{\pm}_{r}(0)|\ge \frac c2|B_{r}(0)|.
\ee
Moreover, without loss of generality, we assume that $\tilde
B=\{|y|\le 1\}$ (this is always true up to a translation and 
dialation)  and take $x\in \{|y|\le 1\}$ such that $k\tilde d(x)< \zeta d(x)$. Le us take now a system of coordinates with origin in the center of $\tilde B$ such that $|x|=-x_1=1-\tilde d(x)$. 
It follows that $\Omegat(x)\cap B_{k\tilde d(x)}(0)=\Sigma \cap B_{k\tilde d(x)}(0)$ and
\begin{align}\label{cip}
 \int_{{\Omega}(x)\setminus \tilde B}\frac{dy}{|x-y|^{N+2s}}\ge  \int_{{\Omega}(x)\cap B_{k\tilde d(x)}(x) \setminus \{y_1>-1\}}\frac{dy}{|x-y|^{N+2s}}=\int_{{\Sigma}\cap B^-_{k\tilde d(x)}(0) \setminus \{z_1>-\tilde d(x)\}}\frac{dz}{|z|^{N+2s}}
\\
\ge\tilde d(x)^{-n-2s}|{\Sigma}\cap B^-_{k\tilde d(x)}(0) \setminus \{z_1>-\tilde d(x)\}|\nonumber.
\end{align}
We get that 
\begin{align}\label{ciop}
|{\Sigma}\cap B^-_{k\tilde d(x)}(0) \setminus \{z_1>-d(x)\}|\ge |{\Sigma}\cap B^-_{k\tilde d(x)}(0)|-|\{-d(x)\le z_1<0\}\cap B^-_{k\tilde d(x)}(0)|
\\
\ge \frac c2 |B_{k\tilde d(x)}|- (kd(x))^N| B_1(0)\cap A(k) |\ge \frac {\omega_n}{4}c(kd(x))^N,\nonumber
\end{align}
where we have used \eqref{auxilia} and the definition of $k$ in the last two inequality respectively. Putting together \eqref{cip} and \eqref{ciop} and recalling the condition  $k\tilde d(x)< \zeta d(x)$, we deduce that 
\[
\int_{{\Omega}(x)\setminus \tilde B}\frac{dy}{|x-y|^{N+2s}}\ge c_1 \tilde d(x)^{-2s} \ \ \ \mbox{for all $x\in \tilde B$ with $k\tilde d(x)\le\zeta$dist$(\tilde B, \Omega)$}.
\]
This, together with the definition of $j(x)$, completes the proof of the Lemma.
\EEE
\end{proof}

%The following is a key Lemma for the following results.
\begin{lemma}[Refined Maximum Principle]\label{aap} Assume \eqref{alpha} and \eqref{Sigmadef}.
Let $\lambda>0$, $0 \leq \EEE f\in C(\Omega)$, and assume
that  $u\in \LSC_b(\Omega)$\EEE,  with
$u>0$ in $\Omega$ and $u=0$ on $\partial\Omega$, satisfies 
\[
h(x)u(x) +\mathcal{L}_s(\Omega(x),u(x) ) \ge \lambda u(x) +f(x) .
\]
Moreover, let  $v\in USC_b(\Omega)$\EEE, with $v\le 0$
on $\partial \Omega$, satisfy
 \[
 h(x)v(x) \EEE+\mathcal{L}_s(\Omega(x),v(x) \EEE)\le \lambda v(x) \EEE.
\]
If $f$ is non trivial then $v\le0$. If $f\equiv0$ and there exists $x_0\in\Omega$ such that $v(x_0)>0$ then $v=tu$ for some $t>0$.
\end{lemma}
\begin{proof}
Let $z_t=v-tu$ for $t>0$. Then, thanks to Corollary \ref{ellipticdif} we have  that $z_t$ satisfies
\be\label{25.6}
h(x)z_t(x)+\mathcal{L}_s(\Omega(x),z_t(x)\EEE)\le \lambda z_t(x).
\ee
Notice that for all  $\rho>0$ and any $t>0$ such that
$$t> \frac{\sup_{d(x)>\frac{\rho}{2}}v(x)}{\inf_{d(x)>\frac{\rho}{2}}u(x)}$$
(recall that $u>0$ in $\Omega$) we have that
\[
\{x\in\Omega \ : \ d(x)\ge\rho \}\subset \{x\in\Omega \ : \ z_t<0 \}.
\]  
We now use \EEE Lemma \ref{localization} to restrict \eqref{25.6}
to \EEE  $\Omega_\rho=\{x\in\Omega \ : \ d(x)<\rho \}$ and get
\be\label{25.6bis}
j(x)z_t+\mathcal{L}_s(\Xi(x),z_t)\le \lambda z_t \ \ \ \mbox{in} \ \ \ \Omega_\rho,
\ee
where $j(x)=h(x)+\int_{\tilde{\Omega}(x)\setminus\Omega_\rho
}\frac{dz}{|z|^{N+2s}}$ and $\Xi(x)=\Omega(x)\cup
\Omega_\rho$. Taking $\rho$ such that $\rho<
\left(\frac{\alpha}{\lambda}\right)^{\frac1{2s}}$ and using the
coercivity assumption \eqref{alpha} on $h(x)$, it follows that
$j(x)-\lambda>0$. Then, since $z_t\le 0$ on $\partial\Omega_{\rho}$,
we can apply the comparison principle to \eqref{25.6bis} and \EEE
deduce that $z_t\le0$ in $\Omega_\rho$. This means that $z_t\le 0$ in $\Omega$.\\

Let us focus on the case $0\not = f \EEE \ge0$ and assume, by contradiction, that there exists $x_0\in\Omega$ such that $v(x_0)>0$. Then, up to a multiplication with a positive constant, we have that $v(x_0)>u(x_0)$.\\
Let us set
\[
\tau=\inf\{ t \ : \ z_t\le0 \ \mbox{in} \  \Omega\}
\]
and recall  that $\tau>1$ since $v(x_0)>u(x_0)$. As 
$z_{\tau}\le 0$ we get 
\[
h(x)z_t(x)\EEE+\mathcal{L}_s(\Omega(x),z_t(x)\EEE)\le \lambda
z_t(x)\EEE\le 0 \quad \forall t \geq \tau.
\] 
We can apply the strong maximum principle of Lemma
\ref{strongmax}  to prove that either $z_{\tau}\equiv0$ or
$z_{\tau}<0$. This latter case is not possible since  it 
would  contradict  the definition of $\tau$.  Having that  
$z_{\tau}\equiv 0$ we get  $v_{\tau} := v=\tau  u$. We have 
\[
h(x)v_{\tau}(x)+\mathcal{L}_s(\Omega(x),v_{\tau}(x))\le \lambda v_{\tau} (x)
\]
by assumption and, since  $v_{\tau}= \tau  u$, 
\[
h(x)v_{\tau}(x)+\mathcal{L}_s(\Omega(x),v_{\tau}(x))\ge \lambda v_{\tau}(x)+ f (x). 
\]
By combining these two inequality, using Corollary \ref{ellipticdif}  and recalling that $f$
is nontrivial, we obtain  a contradiction. Hence,  $v\le0$.

Let us now  consider the case $f\equiv0$ and $v(x_0)>0$ for
some $x_0\in\Omega$. Upon multiplying by  a positive
constant, we can assume that $v(x_0)>u(x_0)$. Following exactly the
same argument and notation of the previous step we obtain that either
$z_{\tau}\equiv0$ or $z_{\tau}<0$. The latter option again \EEE
leads to a contradiction. Hence,  $z_{\tau}\equiv0$, which
corresponds to the assertion. \EEE %. %implies the required result.
\end{proof}

\begin{teo}\label{approximation} Assume \eqref{alpha}-\eqref{Sigmadef}.
Given $\lambda>0$ and a nonzero $0\le f\in C(\Omega)$ satisfying \eqref{fcon}, let us assume that there exists  $0\le  u\in \LSC_b(\Omega)$ \EEE such that 
\[
\begin{cases}
h(x)u(x)\EEE+\mathcal{L}_{s} (\Omega(x),u(x)\EEE)\ge\lambda u(x)\EEE+ f(x) \qquad & \mbox{in } \Omega,\\
 u(x) = 0 & \mbox{on }  \partial \Omega.
\end{cases}
\]
Then, for any $\mu\le \lambda$ and $ |g|\le f$, there exists a solution to 
\be\label{24-6bis}
\begin{cases}
h(x)v(x)\EEE+\mathcal{L}_{s} (\Omega(x),v(x)\EEE)= \mu v(x)\EEE+g(x) \qquad & \mbox{in } \Omega,\\
 v(x) = 0 & \mbox{on }  \partial \Omega.
\end{cases}
\ee
 If moreover $g$ is nonnegative and nontrivial then $v>0$.\EEE
\end{teo}
\begin{proof}
Let us set $v_0=0$ and recursively \EEE define the sequence $\{v_n\}$ of solutions to
\[
\begin{cases}
h(x)v_n(x)\EEE+\mathcal{L}_{s} (\Omega(x),v_n(x))=\mu v_{n-1}+ g(x) \qquad & \mbox{in } \Omega,\\
 v_n(x) = 0 & \mbox{on }  \partial \Omega.
\end{cases}
\]
Notice that the existence of each  $v_n$ is ensured by
Theorem \ref{existence}. We now  prove that $ |v_n|\le u$
by induction on $n$.  %, for any $n=1,2,\cdots$. 
 Let $n= 1$. As $|g|\leq f$ the comparison principle from
Corollary \ref{ellcomp} ensures that $|v_1|\leq u$.  
%there is nothing to prove. %Then assume that the statement
                                %is true for $n-1$ and let us prove it
                                %for $n$. 
 Assume that $|v_{n-1}| \leq u$. Since $|\mu v_{n-1}+ g(x)|\le \lambda u(x)+ f(x)$, we can use the comparison principle (see Corollary \ref{ellcomp}) to deduce that $|v_{n}|\le u$.  In case $g\ge0$ a similar argument shows that $0\le v_{n}\le v_{n+1}$.\\
%Thanks to Lemma \ref{differenza}, the function $w=u-v_1$ satisfies 
%\[
%h(x)w(x)\EEE+\mathcal{L}_{s} (\Omega(x),w(x)\EEE)=\lambda u(x)\EEE-\mu v_{n-1}(x)\EEE+ f(x)-g(x)\ge 0 \qquad  \mbox{in } \Omega,
%\]
%where the inequality comes from the assumption on $\mu$ and $g$ and
%the fact that we already know that $|v_{n-1}|\le u$. Since $w=0$ on
%$\partial\Omega$,  we apply again Lemma \ref{differenza}  (da
%anticipare e scrivere nella versione ellittica) \EEE to deduce that
%$w\ge0$. Similarly we show that $u+v_n\ge0$. This proves \EEE that $|v_n|\le u\le C d(x)^{2s}$. 
This implies that $|\mu v_{n-1}+ g(x)|\le \lambda u+f \le C d(x)^{\eta_f-2s}$, where we have used assumption \eqref{fcon} for the last inequality.
Using Lemma \ref{barriercone}, we can conclude that there exist a
large $Q$ (independent of $n$) such that $\overline{l}(x)=Qd(x)^{\eta}$, with $\eta=\min\{\bar \eta, \eta_f\}$, solves
\[
h(x)\overline{l}(x)+\mathcal{L}_{s} (\Omega(x),\overline{l}(x))\ge \mu v_{n-1}+ g(x) \qquad  \mbox{in } \Omega.
\]
Thanks again to the comparison principle, we deduce that $v_n\le
Qd(x)^{\eta}$. Similarly, it follows that $v_n\ge- Qd(x)^{\eta}$.  Let us consider then the half-relaxed limits of the sequence $v_n$
\[
\overline{v}(x)=\sup\{\limsup_{n\to\infty} v_n(x_n) \ : \ x_n\to x\}, \ \ \ \ \underline{v}(x)=\inf\{\liminf_{n\to\infty} v_n(x_n) \ : \ x_n\to x\}.
\]
Notice that by construction both $\overline{v}$ and $\underline{v}$ vanish on $\partial\Omega$. Taking advantage of Lemma \ref{stability}, we deduce that $\overline{v}$ and $\underline{v}$ are respectively sub and super-solution to \eqref{24-6bis}. Moreover, Corollary \ref{ellipticdif} implies that $w=\overline{v}-\underline{v}$ solves
\[
h(x)w(x)+\mathcal{L}_{s} (\Omega(x),w(x))\le \mu w(x) \ \  \mbox{ in }\ \ \Omega, \ \ \mbox{ and } \ \ w=0 \ \ \mbox{ on } \ \ \partial\Omega.\\
\]
Since $u$ satisfies
\[
h(x)u(x)+\mathcal{L}_{s} (\Omega(x),u(x))\ge\mu u(x)+ f(x) \qquad  \mbox{in } \Omega,\\
\]
we  may   use Lemma \ref{aap} to conclude that $w\le0$, namely
$\overline{v}\le\underline{v}$. Due to the natural order between the
two functions, we deduce that $v=\overline{v}=\underline{v}$ is a
viscosity solution to \eqref{24-6bis}. If $g \geq 0$ and not trivial,
by  using the strong maximum principle we easily deduce that $v>0$. \EEE
\end{proof}

Let us assume that the  nontrivial  $0\le f\in C(\Omega)$ 
satisfies \eqref{fcon} and  recall the definition of \EEE the set
\[
E_f=\{\lambda\in\mathbb R \ : \ \exists v\in C(\overline{\Omega}), \
v>0 \mbox{ in } \Omega, \ v=0 \mbox{ on } \partial\Omega, \
\mbox{such that}  \ hv+\mathcal{L}_s(\Omega, \EEE v)= \lambda v+f\}.
\]
Moreover, let \EEE
\be\label{28.06}
\lambda_f \EEE =\sup \ E_f.
\ee
As we shall see, $\lambda_f$ does not depend on the particular choice of $f$. By definition and thanks to Theorem \ref{approximation} we deduce that
\[
\mbox{if} \ \ \ g\le f   \ \ \ \mbox{then} \ \ \ \lambda_g\le \lambda_f.
\]
The following Lemma shows us that $\lambda_f$ is finite and that $E_f$ is a left semiline. %has no holes.
\begin{lemma}[]\label{acotado}
Assume \eqref{alpha}-\eqref{Sigmadef} and that $0\le f\in C(\Omega)$ is nonzero and  satisfies \eqref{fconintro}. Then, 
$\lambda_f$ is positive and finite and $E_f$ is a left semiline with
$E_f \not = \Rz$.
\end{lemma}
\begin{proof}
Notice that for any $$\dys \lambda\in \left(-\infty,\frac{\alpha}{\mbox{diam}(\Omega)^{2s}}\right)$$ the operator 
\[
u\mapsto \EEE [h(x)-\lambda]u +\mathcal{L}(\Omega(x),u)
\]
fulfills assumptions \eqref{ostationary}-\eqref{alpha}. We can apply
the existence results and the strong maximum principle of the previous
chapter to deduce that
$(-\infty,\frac{\alpha}{\mbox{diam}(\Omega)^{2s}})\subset E_f$. 
 Moreover,  if
$\lambda\in E_f$, Theorem \eqref{approximation} assures that any
$\mu<\lambda$ belongs to $E_f$ as well. This proves that $E_f$  is
 a left semiline.  %too.\\ 

To show that $E_{f}\neq \mathbb{R}$ let us take $\lambda<\lambda_f $.  Since $E_f$  is  a left semiline, there exists some $v\in C(\overline \Omega)$ with $v=0$ on $\partial\Omega$ and strictly positive in $\Omega$, such that \EEE
\[
h(x)v(x)+\mathcal{L}_s(\Omega(x), v(x))= \lambda v(x)+f(x) \ \ \ \mbox{in the viscosity sense in } \ \Omega.
\]
 Now we want to \emph{restrict} this inequality to  a ball
$B\subset\subset\Omega$ such that $f>0$ in $B$\EEE. In order to do it,  for any $x\in
B$, we   define $\Xi(x)=\Omega(x)\cap B$. Taking advantage
of the positivity $v$, we can apply Lemma \ref{localization} to $-v$
and  deduce that
\be\label{cipcip}
j(x)v+{\mathcal{L}}_s(\Xi(x),v) \ge \lambda v+ f(x)  \ \ \ \mbox{in the viscosity sense in } \ B,
\ee
where 
\[
j(x)=h(x)+\int_{{\Omega}(x)\setminus B}\frac{dy}{|x-y|^{N+2s}}.
\]

Thanks to Lemma \ref{localization}, we have \EEE
that $j(x)$ satisfies \eqref{alphap} (by possibly changing the
\EEE constants) and that the family $\{\Xi(x)\}$ satisfies the same
kind of assumptions of $\{{\Omega}(x)\}$. Then, for any
positive continuous function $g$ \EEE  with compact support in $B$ there exists a unique viscosity solution to
\[
\begin{cases}
j(x)w(x)\EEE+\mathcal{L}_s(\Xi(x),w(x)\EEE)= g(x) \qquad & \mbox{in } B,\\
 w(x) = 0 & \mbox{on }  \partial B.
\end{cases}
\]
Thanks to the strong maximum principle of Lemma \ref{strongmax} and the fact that $g$ has compact support in $B$, it
follows that $0<g\le C_0 w$ for some positive constant $C_0$. 
  If $C_0<\lambda$, we would get that 
\be\label{ciopciop}
j(x)w(x)\EEE+\mathcal{L}_s(\Xi(x),w(x)) \le  \lambda w(x)\EEE \ \ \ \mbox{in the viscosity sense in } \ B.
\ee
Applying Lemma \ref{aap} to \eqref{cipcip} and \eqref{ciopciop}, it would follow $w\le0$,  which  is a contradiction. 
 This proves that  $\lambda\le C_0$. Since the constant $C_0$ does not depend on $\lambda$, we finally deduce 
\[
\lambda_f  =\sup \ E_f\le C_0.
\]
This  concludes  the proof of the Lemma.
%The next step is to properly combine inequalities \eqref{cipcip} with
%\eqref{ciopciop} to prove that $\lambda\le C_0$. Let us assume by
%contradiction that $C_0<\lambda$ and set $p
%=\sup_B\frac{w}{v}$, where $v$ is the supersolution to
%\eqref{auxilia} associated to $\lambda$. We stress that $p $
%is positive and finite since $v>0$ in $B$ by definition. Using Corollary
%\ref{ellipticdif},  we deduce that 
% Da rivedere da qui \EEE
%\[
%\hat{h}(p  v-w)+\hat{\mathcal{L}}_s(p  v-w)\ge p \lambda v -C_0 w\ge w(\lambda-C_0)>0.
%\]
%Then since $v>0=w$ on $\partial B$ and thanks to the strong maximum principle (Lemma \ref{strongmax}), we deduce that $p  v-w>0$ in $B$. This contradicts the definition of $\pi$.\\
\end{proof}
% \begin{teo}\label{mussaka}
% There exists a unique $\overline{\lambda}$ such that 
% \begin{equation}\label{eigenproblem}
% \begin{cases}
% {h}(x)v+{\mathcal{L}}_s (\Omega(x), v)=\overline{\lambda} v\qquad & \mbox{in } \Omega,\\
%  v(x) = 0 & \mbox{on }  \partial \Omega,
% \end{cases}
% \end{equation}
% admits a nontrivial (strictly) positive solution. Such a solution is unique up to a multiplicative constant. Moreover, for any $f\in C(\Omega)$ that satisfies \eqref{fconintro}, we have that $\overline{\lambda}=\sup E_f$.
% \end{teo}

Let us provide now the proof of the  well-posedness of the
first-eigenvalue problem.
\begin{proof}[Proof of Theorem \ref{mussaka}] We split the
  argument into subsequent steps. 

\textbf{Step 1:} Let us show at first that for a given $f\in C(\Omega)$ that satisfies \eqref{fconintro} and the additional condition
\be\label{temp}
f(x)\ge \theta>0 \ \ \ \mbox{in} \ \ \ \Omega,
\ee
problem \eqref{eigenproblem} admits a solution $v_f>0$ with $\lambda_f:=\sup E_f$.
Let $\{\lambda_n\}$ be a sequence that converges to ${\lambda_f}$ and
consider the associate sequence $\{v_n\}\subset C(\overline{\Omega})$,
$v_n>0$ of solutions to
\be\label{aug}
\begin{cases}
h(x)v_n(x)+\mathcal{L}_{s} (\Omega(x),v_n(x))=\lambda_n v_n(x)+ f(x) \qquad & \mbox{in } \Omega,\\
 v_n(x) = 0 & \mbox{on }  \partial \Omega.
\end{cases}
\ee
We first claim that $\|v_n\|_{\elle{\infty}}\to\infty$. Indeed, let us
assume by contradiction that there exists $k>0$ such that $ |v_n| \le k$. Thanks to Lemma \ref{barriercone}, we deduce that there exists $Q=Q(k)$ such that the function $\overline{l}(x)=Qd(x)^{\eta}$, with $\eta=\min\{\bar \eta, \eta_f\}$, solves in the viscosity sense
\[
h(x)\overline{l}(x)+\mathcal{L}_{s} (\Omega(x),\overline{l}(x))\ge \lambda_n v_n(x)+ f(x) \quad  \mbox{in } \Omega\,\quad \forall \ n>0 .
\]
Thanks to Corollary \ref{ellcomp} and the sign of $v_n$, we have that
\be\label{8mar}
0<v_n(x)\le Q d^{\eta}(x),
\ee
where the right-hand side does not depend on $n$.  
Using Lemma \ref{stability}, we deduce that $\underline{v}(x)=\inf\{\liminf_{n\to\infty} v_n(x_n) \ : \ x_n\to x\}$ is a supersolution to \EEE
%Using now Theorem \ref{regelliptic} and Lemma \ref{ascoli}, we deduce that there exists $v_{\infty}$ such that $v_n\to v_{\infty}$ uniformly in $\Omega$ (up to a subsequence) and $v_{\infty}=0$ on $\partial\Omega$. Moreover, Lemma \ref{stability} allows us to pass to the limit in \eqref{aug} and conclude that $v_{\infty}$ solves
\be\label{29-11}
\begin{cases}
h(x)v(x)+\mathcal{L}_{s} (\Omega(x),v(x))= \lambda_f \EEE v(x)+ f(x) \qquad & \mbox{in } \Omega,\\
 v(x) = 0 & \mbox{on }  \partial \Omega.
\end{cases}
\ee
 
Since $v_n> 0$ also $\underline{v}\ge0$ in $\Omega$ and Lemma \ref{strongmax} assures
that $\underline{v}>0$ in $\Omega$. Then, Theorem \ref{approximation}
provides   a solution $v_{\infty}>0$ to \eqref{29-11}. \EEE Taking  $\epsilon>0$ such that $f\ge
\frac12f+\epsilon v_{\infty}$, which is possible thanks to \eqref{temp}, it follows that $\tilde{v}_{\infty}=2v_{\infty}$ satisfies
\[
h(x)\tilde{v}_{\infty}(x)+\mathcal{L}_{s} (\Omega(x),\tilde{v}_{\infty}(x))\ge(\lambda_f+\epsilon) \tilde{v}_{\infty}(x)+ f(x) \qquad  \mbox{in } \Omega.
\]
 Taking again advantage of Theorem \ref{approximation} we reach a
contradiction with  respect to  the definition of
$\lambda_f$. \EEE

 We have then proved  that
$\|v_n\|_{\elle{\infty}}\to\infty$. Setting
$u_n=v_n\|v_n\|_{\elle{\infty}}^{-1}$  we obtain  that
\[
\begin{cases}
h(x)u_n(x)+\mathcal{L}_{s} (\Omega(x),u_n(x))=\lambda_n u_n(x)+
\frac{ f(x) }{\|v_n\|_{\elle{\infty}}} \qquad & \mbox{in } \Omega,\\
 u_n(x) = 0 & \mbox{on }  \partial \Omega.
\end{cases}
\]
 Since $0<u_n\le 1$, we deduce as in \eqref{8mar} that $u_n\le Qd^{\eta}(x) $ and then
\[
\overline{u}(x)=\sup\{\limsup_{n\to\infty} u_n(x_n) \ : \ x_n\to x\}, \ \ \ \ \underline{u}(x)=\inf\{\liminf_{n\to\infty} u_n(x_n) \ : \ x_n\to x\},
\]
are well defined and vanish on the boundary. Lemma \ref{stability} assures us that they are sub and super solution to 
\be\label{paris}
\begin{cases}
h(x)u(x)+\mathcal{L}_{s} (\Omega(x),u(x))= \lambda_f \EEE  u(x)\qquad & \mbox{in } \Omega,\\
 u_{\infty}(x) = 0 & \mbox{on }  \partial \Omega.
\end{cases}
\ee

By definition of $u_n$, there exists a sequence of points $\{x_n\}\subset\Omega$ such that $u_n(x_n)=1$. Thanks to the uniform bound $u_n\le Qd^{\eta}(x) $, we deduce that there exists $\Omega'\subset\subset \Omega$ and that $\{x_n\}\subset\Omega'$. Using Lemma \ref{regelliptic}, it follows that
\[
\|u_n\|_{C^{\gamma}(\Omega')}\le \tilde C(C,s,\zeta,d(\Omega'', \Omega'), Q\|d\|_{\elle{\infty}}^{\eta}),
\]
with $\Omega'\subset\subset \Omega''\subset\subset \Omega$. Then, up
to a not relabeled subsequence, $u_n$ uniformly converges to
a continuous function $u\in C(\overline{\Omega'})$. Furthermore,
$u\equiv \overline{u}\equiv \underline{u}$ in $\Omega'$ and
$u_n(x_n)\to u(\bar x)=1$ for some $\bar x\in \Omega'$. This
entails on the one hand that % provide us with two important pieces of
                             % information. The first one is that 
 $\underline{u}$ is a non negative and nontrivial supersolution to
\eqref{paris}, so that Lemma \ref{strongmax} implies
$\underline{u}>0$. On the other hand,  we obtain that there exists  $\bar x\in \Omega'$ such that $\overline{u}(\bar x)>0$.

By applying again Lemma \ref{aap} directly to $\overline u$ and $\underline
u$, we conclude that there exist $t>0$ such that $\overline u= t
\underline u$.  This implies that both functions are continuous.
Moreover, since $t>0$ we deduce that both $\underline u$ and
$\overline u$ are at the same time sup- and super-solutions. Hence, $\underline u$ and
$\overline u$  are eigenfunctions related to $\lambda_f$.  \EEE

Now we want to get rid of assumption \eqref{temp}. Notice that we used it to show that $\|v_n\|_{\elle{\infty}}$ must diverge.
Then, we have to prove that $\|v_n\|_{\elle{\infty}}\to\infty$,
assuming that the nontrivial positive continuous function $f$ solely
satisfies \eqref{fcon}. Again, let us argue by contradiction and
suppose that $\|v_n\|_{\elle{\infty}}\le k$.  This would lead again to
the existence of a non trivial $v_{\infty}$ solution to \eqref{29-11}. Setting $g=\sup\{f,\theta\}$, the previous argument provides us with $\lambda_g>0$ and $v_g>0$ solution to \eqref{eigenproblem}. 
Since $f\le g$, by construction  we deduce that  $\lambda_f\le
\lambda_g$. Assume that $\lambda_f< \lambda_g$ and take $\mu\in 
(\lambda_f,\lambda_g)$.  Then, thanks to the definition of
$\lambda_g$, the fact that $f\le g$ and  by   using Theorem \ref{approximation}, it results that the following problem
\[
\begin{cases}
h(x)z(x)+\mathcal{L}_{s} (\Omega(x),z(x))=\mu z(x) + f(x) \qquad & \mbox{in } \Omega,\\
 z(x) = 0 & \mbox{on }  \partial \Omega,
\end{cases}
\]
admits a positive solution.  This however  contradicts the
fact that $\lambda_f$ is a supremum. On the other hand, if $\lambda_f=
\lambda_g$, Lemma \ref{aap}, applied to $v_{\infty}$ and $v_g$, would
imply $v_g\le0$,  which is again  a contradiction.

\textbf{Step 2:} Let us assume that there exists another couple $(\mu, w)\in \mathbb (0,\infty)\times C(\overline \Omega)$ that solves \eqref{eigenproblem} with $w>0$. If $\mu< \lambda_f$, we deduce that there exists
$\lambda\in (\mu,\lambda_f)$ and, by definition of $\lambda_f$ and
Lemma \ref{acotado}, a function $u\in C(\overline{\Omega})$, with
$u>0$ in $\Omega$,  solving 
\[
\begin{cases}
{h}(x)u(x)+{\mathcal{L}}_s (\Omega(x),u(x))=\lambda u(x)+ f(x) \qquad & \mbox{in } \Omega,\\
 u(x) = 0 & \mbox{on }  \partial \Omega.
\end{cases}
\]
On the other hand, since $\mu<\lambda$ and $w>0$ we have that
\[
\begin{cases}
{h}(x)w(x)+{\mathcal{L}}_s (\Omega(x),w(x))\le\lambda w(x) \qquad & \mbox{in } \Omega,\\
 w(x) = 0 & \mbox{on }  \partial \Omega.
\end{cases}
\]
Being in the same setting of Lemma \ref{aap}, we deduce that $w\le0$,
 which  is a contradiction. 
%Let us stress at this point that this argument implies that $\lambda_f$ does not depend on $f$. Indeed take $f,g$ that satisfy \eqref{26-6} and such that $g\le f$. Then thanks to Theorem \eqref{approximation}

 This proves that  $ \mu\ge\lambda_f$. Assume by contradiction, that $\mu> \lambda_f$. Take $\epsilon>0$ small enough in order to have $\lambda_f<\mu-\epsilon$ and a nonnegative nontrivial continuous function $g(x)$ such that $\epsilon w \ge g$ in $\Omega$. It follows that $w$ solves
\[
{h}(x)w(x)+{\mathcal{L}}_s (\Omega(x),w(x))\ge(\mu-\epsilon) w(x) + g(x) \qquad \mbox{in } \Omega.
\]
Using Theorem \ref{approximation} we deduce that there exists $w_g$ solution to 
\[
\begin{cases}
{h}(x)w_g(x)+{\mathcal{L}}_s (\Omega(x), w_g(x))=(\mu-\epsilon) w_g(x)+ g(x) \qquad & \mbox{in } \Omega,\\
 w_g(x) = 0 & \mbox{on }  \partial \Omega.
\end{cases}
\]
Applying the refined comparison principle of Lemma \ref{aap} between
$v_f$ and $w_g$ it follows that $v_f\le 0$,  which  is a
contradiction.  We eventually conclude  that
$\lambda_f=\mu$. This also implies that
$\lambda_f=\lambda_g=\overline{\lambda}$ for all $f,g$ that satisfy
\eqref{fconintro}.

\textbf{Step 3.} In this last step we show that solutions of
\eqref{eigenproblem} are unique up to a multiplicative
constant. Assume that $w$ is a nontrivial solution to
\eqref{eigenproblem} and let $v_f>0$ be the solution provided by Step
1. Since $w$ is nontrivial we can always assume, up to a
multiplication with a  (not  necessarily positive) constant, that $w(x_0)>0$. Then, we can use the second part of Lemma \ref{aap} to conclude that $w=t v_f$ for some constant $t$.
\end{proof}

\begin{proof}[Proof of Theorem \ref{helmholtz}] 
Since $\lambda< \overline{\lambda}$, thanks to the assumptions on $f$
and using characterization \eqref{28.06}, we deduce that there exists
a function $v>0$ in $\Omega$  solving 
\[
\begin{cases}
{h}(x)v(x)+{\mathcal{L}}_s (\Omega(x),v(x))=\lambda v(x)+ |f(x)| \qquad & \mbox{in } \Omega,\\
 v(x) = 0 & \mbox{on }  \partial \Omega.
\end{cases}
\]
Clearly, $v$ is a supersolution for \eqref{helmeq} and we can take
advantage of Theorem \ref{approximation} to conclude that
\eqref{helmeq} admits a solution. To deal with uniqueness we assume that \eqref{helmeq} has two solutions $v$ and $z$ and set $w=v-z$. Using Corollary \ref{ellipticdif} it follows that $w$ solves
\[
\begin{cases}
{h}(x)w(x)+{\mathcal{L}}_s (\Omega(x),w(x))=\lambda w(x)  \qquad & \mbox{in } \Omega,\\
 w(x) = 0 & \mbox{on }  \partial \Omega.
\end{cases}
\]
Applying Lemma \ref{aap} to $v$ and $w$, we conclude that $w\le
0$. The same conclusion holds for $ -w=  z-v$, 
so that $w=0$ and $v\equiv z$.  % we are done.
\end{proof}

%Let us consider now two families of sets  $x\mapsto \Omega(x)$ and $x\mapsto \Sigma(x)$ that satisfy \eqref{omegax} and define
%\[
%h_{\Omega}(x)=\int_{\tilde\Omega(x)^c}\frac{dz}{|z|^{N+2s}} \ \ \ \mbox{and} \ \ \ h_{\Sigma}(x)=\int_{\tilde\Sigma(x)^c}\frac{dz}{|z|^{N+2s}}.
%\]
%\[
%\begin{cases}
%{h}_{\Omega}(x)v+{\mathcal{L}}_s (\Omega(x), v)={\lambda}_{\Omega} v\qquad & \mbox{in } \Omega,\\
% v(x) = 0 & \mbox{on }  \partial \Omega,
%\end{cases}
%\ \ \ \ \  \ \ \ 
%\begin{cases}
%{h}_{\Sigma}(x)v+{\mathcal{L}}_s (\Sigma(x), v)={\lambda}_{\Sigma} v\qquad & \mbox{in } \Omega,\\
% v(x) = 0 & \mbox{on }  \partial \Omega,
%\end{cases}
%\]
%\begin{lemma}
%Let us consider two set valued functions $x\mapsto \Omega(x)$ and $x\mapsto \Sigma(x)$. Assume that both of them satisfies assumption \eqref{omegax} and that moreover
%\[
%\Omega(x) \le \Sigma(x) \ \ \ \mbox{for any } \ \ x\in\Omega.
%\]
%Then we have 
%\[
%\lambda_1(\Omega(x))\ge \lambda_1(\Sigma(x)).
%\]
%\end{lemma}

\section{Asymptotic analysis}\label{Asymptotic}

In this  last  section, we address the large time behavior of the solution to \eqref{parabolic}. 
%Our first results says that, if the source of the problem goes to zero fast enough as $t$ diverges, then the solution to \eqref{parabolic} converges uniformly to zero for large times. 

\begin{teo}\label{totoinfty} Let us assume \eqref{alpha}-\eqref{Sigmadef}, that $ g_0 \in C(\Omega)$ with supp$( g_0
  )\subset\subset\Omega$,  and that there exist constants $\eta_g,C_g>0$ such that the continuous nonnegative function $g:(0,\infty)\times\Omega\to\mathbb R^{+}$ satisfies 
\be\label{gcon}
g(t,x)d(x)^{2s-\eta_g}e^{\lambda t}\le C_g \ \ \ \mbox{for some} \ \ \ \lambda<\overline{\lambda},
\ee
where $\overline{\lambda}$ is the first eigenvalue provided by Theorem \emph{\ref{mussaka}}.
Then, if $w\in C([0,\infty)\times\overline{\Omega})\cap
L^{\infty}((0,\infty)\times\Omega)$ satisfies in the viscosity sense 
\[
-g(x,t)\le \partial_tw(t,x)+h(x)w(t,x)+\mathcal{L}_{s}(\Omega(x),w(t,x)) \le g(x,t)   \ \ \ \mbox{in} \  (0,\infty)\times\Omega,
\]
coupled with boundary and initial conditions
\[
\begin{cases}
 w(t,x) = 0 & \mbox{on } (0,T)\times\partial\Omega,\\
 w(0,x) = g_0(x) & \mbox{on }   \Omega,
\end{cases}
\]  
 one has  $|w(x,t)|\le Q_{\lambda}d(x)^{\tilde
  \eta}e^{-\lambda t}$, for  all  $\tilde \eta\le \min\{\bar
\eta, \eta^g\}$ and  some  $Q_{\lambda} >0$ with
$Q_\lambda \to \infty$  as $\lambda\to \overline \lambda$ . 
\end{teo}
\begin{proof}
Let us consider $\varphi_{\lambda}$  solving 
\begin{equation}\label{barrierup}
\begin{cases}
h(x)\varphi_{\lambda}(x)+\mathcal{L}_{s} (\Omega(x),\varphi_{\lambda}(x))= \lambda \varphi_{\lambda}(x) +C_g d(x)^{\eta_g-2s} \qquad & \mbox{in } \Omega,\\
 \varphi_{\lambda}(x) = 0 & \mbox{on } \partial\Omega,\\
 \varphi_{\lambda}(x) > 0 & \mbox{in } \Omega.\\
\end{cases}
\end{equation}
Such a function exists since $\lambda<\overline{\lambda}$ and thanks to the characterization of Theorem \ref{mussaka}.
We want to prove that $\overline w(t,x)=e^{-\lambda t}\varphi_{\lambda}(x)$ solves in the viscosity sense
\[
\partial_t\overline w(t,x)+h(x)\overline w(t,x)+\mathcal{L}_{s}(\Omega(x),\overline w(t,x)) \ge g(x,t) \qquad \mbox{in }  (0,\infty)\times\Omega.
\]
 In order to achieve this,  let $\phi\in C^2(\Omega\times(0,\infty))$ and $(t,x)\in(0,\infty)\times\Omega$ such that $\overline w(t,x)=\phi(t,x)$ and that $\overline w(\tau,y)\ge\phi(\tau,y)$ for $(\tau,y)\in(0,\infty)\times\Omega$. We need to check (see Lemma \eqref{def2}) that for any $B_r(x)\subset \Omega$
\[
 \partial_t  \phi_r( t , x )+ h( x ) \phi_r( t , x )+\mathcal{L}_{s}(\Omega(x), \phi_r(t,x))\ge g(x,t) . 
\]
By construction of $\phi_r$ we have that $\partial_t  \phi_r( t , x
)\ge -\lambda e^{-\lambda t}\varphi_{\lambda}(x)$. Moreover, the
function $\psi^t(y)=\phi e^{\lambda t}$ touches  $\varphi_{\lambda}$
at $x$ from below. Then, we infer that 
\begin{align*}
&
h( x ) \phi_r( t , x )+\mathcal{L}_{s}(\Omega(x),
  \phi_r(t,x))=e^{-\lambda t}\left[ h( x )\psi^t_r( x
  )+\mathcal{L}_{s}(\Omega(x), \psi^t_r( x  )) \right]\\
&\quad \ge e^{-\lambda t} [ \lambda \varphi_{\lambda}(x) +C_g d(x)^{\eta_g-2s}],
\end{align*}
where the last inequality follows from the definition of
$\varphi_{\lambda}$.  By collecting the information obtained  we get that
\begin{align*}
\quad& \partial_t  \phi_r( t , x )+ h( x ) \phi_r( t , x
  )+\mathcal{L}_{s}(\Omega(x), \phi_r(t,x))- g(x,t)\\
&\quad
\ge -\lambda e^{-\lambda t}\varphi_{\lambda}(x) + e^{-\lambda t} [
  \lambda \varphi_{\lambda} +C_g d(x)^{\eta_g-2s}] - g(x,t)\\
&\quad=e^{-\lambda t}d(x)^{\eta_g-2s} [C_g -g(x,t)e^{\lambda t}d(x)^{2s-\eta_g}]\ge0,
\end{align*}
where the last inequality comes from assumption \eqref{gcon}. Similarly, we can prove that $\underline w(t,x)=-e^{-\lambda t}\varphi_{\lambda}(x)$ solves
\[
g(x,t) \le \partial_t\overline w(t,x)+h(t,x)\overline w(t,x)+\mathcal{L}_{s}(\Omega(t,x),\overline w(t,x))  \qquad \mbox{in }  (0,\infty)\times\Omega.
\]

Since  $g_0$ has a compact support   we can assume
$| g_0 |\le\varphi_{\lambda}$,  for otherwise we can
consider  $k\varphi_{\lambda}$ instead of $\varphi_{\lambda}$ for
large $k>0$.  Therefore, we  can  use the comparison principle to deduce that
\[
|w(t,x)|\le e^{-\lambda t}\varphi_{\lambda}(x).
\]
At this point, notice that the right-hand side of the  first  equation in \eqref{barrierup} can be estimated as follows
\[
\lambda \varphi_{\lambda}(x) +C_g d(x)^{\eta_g-2s}\le \lambda \|\varphi_{\lambda}\|_{\elle{\infty}}+C_g d(x)^{\eta_g-2s}\le C_{\lambda,g }d(x)^{\eta_g-2s}. 
\]
This implies that there exists $Q=Q(C_{\lambda,g })$ large enough so
that $\overline l= Qd(x)^{\tilde \eta}$ is a super solution to
\eqref{barrierup} (see Lemma \ref{barriercone}). Then, we can use the
comparison principle to conclude that $\varphi_{\lambda}(x)\le
\overline l$,  which concludes   the proof.
\end{proof}
%\begin{remark}
%A straightforward application of Theorem \eqref{totoinfty} is that any bounded solution to
%\[
% \partial_tw+h(t,x)w+\mathcal{L}_{s}(\Omega(t,x),w) = f(x,t)    \ \ \ \mbox{in} \  (0,\infty)\times\Omega,
%\]
%with $f(x,t) $ satisfying $|f(t,x)|d(x)^{2s}e^{\lambda t}\le C$, for $\lambda$ as in \eqref{gcon} and for some positive $C$, converges to zero as $t\to\infty$.
%\end{remark} 
%Now we want to deal with nontrivial stationary states. 
%The general strategy is: to consider $u(t,x)$, the unique viscosity
%solution to the parabolic problems \eqref{parabolic}, and $v(x)$, the
%unique viscosity solution of the stationary problem \eqref{10:17}; to
%assume that the time dependent \emph{coefficients} ($h(t,x),
%\Omega(t,x), f(t,x)$) of the first suitable converge to the time
%independent ones of the latter; finally to apply Theorem
%\ref{totoinfty} to $w(x,t)=u(x,t)-v(x)$. \\

 We conclude by presenting a proof of  Theorem \ref{lalaguna}.
\begin{proof}[Proof of Theorem \ref{lalaguna}]
Using Lemma \ref{differenza}, it follows that $w(t,x)=u(t,x)-v(x)$ solves in the viscosity sense
\be\label{4-6}
 \partial_tw(t,x)+h(t,x)w(t,x)+\mathcal{L}_{s} (\Omega(x),w(t,x))\le \tilde{f}(x,t) \qquad \mbox{in }  (0,\infty)\times\Omega
\ee
where
\[
\tilde{f}(x,t)=|f(x,t)-f({x} )|+M|h({t} ,{x} )-h({x} )|+2M\int_{\mathbb{R}^N}\frac{|\chi_{\Omegat({t} ,{x} )}-\chi_{\Omegat({x} )}|}{|z|^{N+2s}}dz.
\]
Similarly, we can apply again Lemma \ref{differenza} to $-w(t,x)=v(x)-u(t,x)$ to deduce that 
\be\label{4-6bis}
 \partial_tw+h(t,x)w+\mathcal{L}_{s} (\Omega(x),w)\ge -\tilde{f}(x,t) \qquad \mbox{in }  (0,\infty)\times\Omega.
\ee
Notice now that 
\[
\int_{\mathbb{R}^N}\frac{|\chi_{\Omegat({t} ,{x} )}-\chi_{\Omegat({x} )}|}{|z|^{N+2s}}dz=\int_{|z|\ge\frac{\zeta}{2}d(\bar x)}\frac{|\chi_{\Omegat({t} ,{x} )}-\chi_{\Omegat({x} )}|}{|z|^{N+2s}}dz
\]
\[
\le \frac{C}{d(x)^{N+2s}}|\Omega(t,x)\Delta\Omega(x)|\le  \frac{Ce^{-\lambda t}}{d(x)^{2s-\eta_1}},
\]
where we have used assumptions \eqref{ostationary} and \eqref{omegax} to deduce the equation in the first line, and assumption \eqref{decayomega} to deduce the last inequality in the second line.
Thanks to inequalities \eqref{4-6} and \eqref{4-6bis},  the
assertion follows  by a direct application of Theorem \ref{totoinfty}.
\end{proof}

%%%%%%%%%%%%%%%%%%%%%%%%%%%%%%%%%%%%%%%%%%
\section*{Acknowledgement} 
\noindent S.B. is supported by the Austrian Science Fund (FWF) projects F65, P32788 and FW506004. 
 U.S. is supported by the Austrian Science Fund (FWF)
projects F\,65, W\,1245, I\,4354, I\,5149, and P\,32788 and by the
OeAD-WTZ project CZ 01/2021.  

\EEE

\end{document}